\documentclass[11pt,a4paper,oneside]{amsart}

\usepackage{amsmath} % Lots of maths functionality
\usepackage{amssymb} % Maths symbols
\usepackage{amsthm} % Maths environments: \begin{proof}, etc.
\usepackage[english]{babel} % Language and hyphenation.
\usepackage[font=small,justification=centering]{caption} % More flexibility for captioning figures
\usepackage[nodayofweek]{datetime}
\usepackage[shortlabels]{enumitem} % Change enumeration labelling with \begin{enumerate}[a)] etc.
\usepackage[T1]{fontenc} % For font encoding, to allow accents, copy & paste, inequality signs, etc. to all work nicely.
\usepackage[utf8]{inputenc} % To be loaded after fontenc, also for encoding.
\usepackage{ifthen} % For Dani's \begin{com} environment and \numberedtheorem
\usepackage{mathabx} % Contains \pnest symbol and more. Clashes with accents for \ring
\usepackage{mathtools} % Uses amsmath, fixes quirks and adds functionality.
\usepackage[dvipsnames]{xcolor} % Allow links to have colour. Needs to be before hyperref and tikz-cd
\usepackage[pdftex,  colorlinks=true, pagebackref=true]{hyperref} % Makes references and citations into links
    \hypersetup{urlcolor=RoyalBlue, linkcolor=RoyalBlue,  citecolor=black}
\usepackage{setspace} % Allows \onehalfspacing etc. for changing gaps between lines
\usepackage{tikz-cd} % Commutative diagrams
\usepackage{xfrac} % Nicer quotients 
\usepackage{cleveref}
\renewcommand*{\backref}[1]{}
\renewcommand*{\backrefalt}[4]
{
    \ifcase #1
        No citation in the text.
    \or
        Cited on Page #2.
    \else
        Cited on Pages #2.
    \fi
}

\makeatletter
\def\@tocline#1#2#3#4#5#6#7{\relax
  \ifnum #1>\c@tocdepth % then omit
  \else
    \par \addpenalty\@secpenalty\addvspace{#2}%
    \begingroup \hyphenpenalty\@M
    \@ifempty{#4}{%
      \@tempdima\csname r@tocindent\number#1\endcsname\relax
    }{%
      \@tempdima#4\relax
    }%
    \parindent\z@ \leftskip#3\relax \advance\leftskip\@tempdima\relax
    \rightskip\@pnumwidth plus4em \parfillskip-\@pnumwidth
    #5\leavevmode\hskip-\@tempdima
      \ifcase #1
       \or\or \hskip 1em \or \hskip 2em \else \hskip 3em \fi%
      #6\nobreak\relax
    \dotfill\hbox to\@pnumwidth{\@tocpagenum{#7}}\par% <---- \dotfill -> \hfill
    \nobreak
    \endgroup
  \fi}
\makeatother

\makeatletter
\@namedef{subjclassname@1991}{Mathematical subject classification 1991}
\@namedef{subjclassname@2000}{Mathematical subject classification 2000}
\@namedef{subjclassname@2010}{Mathematical subject classification 2010}
\@namedef{subjclassname@2020}{Mathematical subject classification 2020}
\makeatother

\newtheorem{thm}{Theorem}[section]

\newtheorem{corollary}[thm]{Corollary}
\newtheorem{prop}[thm]{Proposition}

\newtheorem{assumption}[thm]{Assumption}
\newtheorem{question}[thm]{Question}

\newtheorem{thmx}{Theorem}
\newtheorem{corx}[thmx]{Corollary}

\theoremstyle{definition}
\newtheorem{defn}[thm]{Definition}
\newtheorem{remark}[thm]{Remark}
\newtheorem{construction}[thm]{Construction}

\newtheorem{example}[thm]{Example}

\newtheorem*{remark*}{Remark}

    \newtheoremstyle{TheoremNum}
        {8.0pt plus 2.0pt minus 4.0pt}{8.0pt plus 2.0pt minus 4.0pt} %%% space between body and thm
        {\itshape} %%% Thm body font
        {-0.15cm} %%% Indent amount (empty = no indent)
        {\bfseries} %%% Thm head font
        {.} %%% Punctuation after thm head
        { }  %%% Space after thm head
        {\thmname{#1}\thmnote{ \bfseries #3}}%%% Thm head spec
    \theoremstyle{TheoremNum}
    \newtheorem{duplicate}{}

\DeclareMathOperator{\Aut}{\mathrm{Aut}}

\DeclareMathOperator{\Ker}{\mathrm{Ker}}

\DeclareMathOperator{\im}{\mathrm{Im}}

\DeclareMathOperator{\Comm}{\mathrm{Comm}}
\DeclareMathOperator{\Isom}{\mathrm{Isom}}

\newcommand{\cala}{{\mathcal{A}}}
\newcommand{\calb}{{\mathcal{B}}}
\newcommand{\calc}{{\mathcal{C}}}

\newcommand{\calg}{{\mathcal{G}}}

\newcommand{\call}{{\mathcal{L}}}

\newcommand{\cals}{{\mathcal{S}}}
\newcommand{\calt}{{\mathcal{T}}}
\newcommand{\calu}{{\mathcal{U}}}
\newcommand{\calv}{{\mathcal{V}}}

\newcommand{\OO}{\mathrm{O}}

\newcommand{\LM}{\mathrm{LM}}

\newcommand{\CAT}{\mathrm{CAT}}

\DeclareMathOperator{\Ad}{\mathrm{Ad}}

\newcommand{\EE}{\mathbb{E}}
\newcommand{\RH}{\mathbb{R}\mathbf{H}} % Real hyperbolic space
\newcommand{\Covol}{\mathrm{Covol}}
\newcommand{\Vol}{\mathrm{Vol}}
\newcommand{\rank}{\mathrm{rank}}

\newcommand{\ZZ}{\mathbb{Z}}

\newcommand{\RR}{\mathbb{R}}

\usepackage{tikz}
\usetikzlibrary{arrows,quotes}
\tikzstyle{blackNode}=[fill=black, draw=black, shape=circle]

\title{Graphs and complexes of lattices}
\author{Sam Hughes}
\address[Current]{Sam Hughes, Rheinische Friedrich-Wilhelms-Universit\"at Bonn, Mathematical Institute, Endenicher Allee 60, 53115 Bonn, Germany}
\address[Former]{Sam Hughes, School of Mathematical Sciences, University of Southampton, Southampton, SO17 1BJ United Kingdom}
\email{sam.hughes.maths@gmail.com}\email{hughes@math.uni-bonn.de}
\date{\today}

\subjclass{20F67, 20E08, 57M07, 20E34, 57M60}

\begin{document}
\begin{abstract}
We study lattices acting on $\CAT(0)$ spaces via their commensurated subgroups.  To do this we introduce the notions of graphs and complexes of lattices which encode graph and complex of group splittings of $\CAT(0)$ lattices.  Using this framework we characterise irreducible uniform $(\Isom(\EE^n)\times T)$-lattices by $C^\ast$-simplicity and give a necessary condition for lattices in products with a Euclidean factor to be biautomatic.  We also construct non-residually finite uniform lattices acting on products of right-angled buildings and non-biautomatic lattices acting on the product of $\EE^n$ and a right-angled building.
\end{abstract}
\maketitle

\section{Introduction}
Let $H$ be a locally compact group with Haar measure $\mu$.  A discrete subgroup $\Gamma\leq H$ is a \emph{lattice} if the covolume $\mu(H/\Gamma)$ is finite.  We say the lattice \emph{uniform} is $H/\Gamma$ is cocompact and \emph{non-uniform} otherwise.  We say a lattice $\Gamma$ in a product $H_1\times H_2$ is \emph{algebraically irreducible} if no finite index subgroup of $\Gamma$ splits as a direct product of two infinite groups, otherwise we say $\Gamma$ is \emph{reducible}.  We will give a topological interpretation of irreducibility for $\CAT(0)$ lattices in Section~\ref{sec.irr}.  In the uniform case we will simply refer to the lattice as \emph{irreducible}.  Given a pair of locally compact groups $H_1$ and $H_2$ there are a number of basic questions one can ask:
\begin{enumerate}[label=(Q\arabic*)]
    \item Does $H_1\times H_2$ contain algebraically irreducible lattices? \label{q.intro.1}
    \item What are the generic properties of an algebraically irreducible lattice? \label{q.intro.2}
\end{enumerate}

In the classical setting of lattices in semisimple Lie groups and linear algebraic groups over local fields these questions are well studied.  Indeed, there are deep theorems such as the Margulis normal subgroup theorem, super-rigidity theorem, and the arithmeticity theorem \cite{MargulisBook}.

The non-classical setting is more complicated and was initiated by studying lattices in the full automorphism group of a locally-finite polyhedral complex.  A striking example of the non-classical setting is given by the work of Burger and Mozes \cite{BurgerMozes1997,BurgerMozes2000struct,BurgerMozes2000lat}.  The authors constructed torsion-free simple groups which could be realised as cocompact irreducible lattices in a product of automorphism groups of locally-finite trees.  More generally, we can consider lattices acting on $\CAT(0)$ spaces; a class which encompasses symmetric spaces, non-positively curved manifolds, Euclidean and hyperbolic buildings, and more.  The reader is referred to \cite{BridsonHaefligerBook} for a comprehensive introduction to the theory.

\begin{assumption}
Throughout this paper, all spaces are assumed to have finite dimensional Tits boundary and all actions of groups on graphs or polyhedral complexes are assumed to be \emph{admissible}.  That is, each element of a group fixes pointwise each cell it stabilises. 
\end{assumption}

\begin{remark}
    An action of a group $G$ on a tree $\calt$ is admissible if and only if $G$ acts on the tree without inversions.  Note that for any group action on a polyhedral complex $X$, the action can be made admissible by barycentric subdividing $X$.
\end{remark}

% Thus, one should find a class of spaces which contain the exciting phenomena to be found in products of polyhedral complexes whilst enjoying a strong geometric grounding.  The answer was to be found in the notion of non-positive curvature or $\CAT(0)$ spaces.  The theory encompasses symmetric spaces, non-positively curved manifolds, Euclidean and hyperbolic buildings, and more \cite{BridsonHaefligerBook}.  The reader is referred to \cite{BridsonHaefligerBook} for a comprehensive introduction to the theory.

We say a locally compact group is \emph{irreducible} if no finite index open subgroup splits non-trivially as a direct product.  We say a metric space is \emph{irreducible} if it does not decompose as a a non-trivial direct product.  

A systematic study of the full isometry groups of $\CAT(0)$ spaces and their lattices was undertaken by Caprace and Monod \cite{CapraceMonod2009a,CapraceMonod2009b,CapraceMonod2019}.  The authors showed in \cite[Theorem~1.6]{CapraceMonod2009a}, that under mild hypotheses on a $\CAT(0)$ space $X$, there is a finite index subgroup of $H\leq\Isom(X)$ which splits as
\begin{equation}  
H\cong\Isom(\EE^n)\times S_1\times\dots\times S_p \times D_1\times\dots\times D_q,
\end{equation}
for some $n,p,q\geq 0$, where each $S_i$ is an almost connected simple Lie group with trivial centre and each $D_j$ is a totally disconnected irreducible group with trivial amenable radical.  Moreover by \cite[Addendum~1.8]{CapraceMonod2009a}, $X$ itself splits as
\begin{equation}
X= \EE^n\times X_1\times\dots\times X_p\times Y_1\times\dots\times Y_q 
\end{equation}
where each $X_i$ is an irreducible symmetric space of non-compact type and each $Y_j$ is an irreducible minimal $\CAT(0)$-space.

Taking these decompositions as a starting point motivates a new approach towards $\CAT(0)$ groups, that is, understanding the lattices in each of the factors individually and then how the factors interact.  The latter question is the central goal of this paper:  To provide a combinatorial framework for studying lattices in products of irreducible $\CAT(0)$ spaces and deduce properties of the algebraically irreducible lattices.  To this end we introduce the notion of a \emph{graph of lattices} (Definition~\ref{def.gol}) with fixed locally-finite Bass--Serre $\calt$ (we often assume that its automorphism group is non-discrete and unimodular, these are essentially non-degeneracy conditions so that there are uniform tree lattices).  Note that in the case of a product of two trees a similar construction was considered by Benakli and Glasner \cite{BenakliGlasner2002}.

\begin{defn}
Let $P$ be a property of a group.  We say a group $\Gamma$ is \emph{covirtually $P$} if there exists a finite normal subgroup $F\trianglelefteq \Gamma$ and $\Gamma/F$ has the property $P$.  Let $H$ be a locally compact group $H$.  We say a pair $(A,\psi)$ is \emph{covirtually a (uniform) $H$-lattice} if $A$ is a group and there exists a homomorphism $\psi\colon A\to H$ such that $\im(\psi)$ is a (uniform) $H$-lattice and $\ker(\psi)$ is finite.  If $(A,\psi)$ is covirtually an $H$-lattice, we define the \emph{covolume of $A$} to be $\Covol(A)=\mu(H/\psi(A))$, where $\mu$ is the Haar measure on $H$.
\end{defn}

Roughly a graph of lattices is a graph of groups such that all local groups are commensurable finite-by-$\{H$-lattices$\}$, equipped with a morphism $\psi$ to $H$.    We use this to study lattices in the product of $T:=\Aut(\calt)$ and closed subgroups $H$ of the isometry group of a fairly generic $\CAT(0)$ space.  We prove a structure theorem for $(H\times T)$-lattices.  That is, we show every $(H\times T)$-lattice gives rise to a graph of $H$-lattices and conversely, we give necessary and sufficient conditions for a graph of $H$-lattices to be an $(H\times T)$-lattice.  It should be mentioned that versions of this theorem are almost certainly known to experts.  Our main new contribution is the formalism and the plethora of applications to follow.

\begin{duplicate}[Theorem~\ref{thm.structure}]
Let $X$ be a proper $\CAT(0)$ space and let $H=\Isom(X)$ contain a uniform lattice.  Let $(A,\cala,\psi)$ be a graph of $H$-lattices with locally-finite Bass--Serre tree $\calt$, and fundamental group $\Gamma$.  Suppose $T=\Aut(\calt)$ is non-discrete and unimodular.
\begin{enumerate}
    \item Assume $A$ is finite.  If each local group $(A_\sigma,\psi|_{A_\sigma})$ is covirtually a uniform $H$-lattice, and the kernel $\Ker(\psi|_{A_\sigma})$ acts faithfully on $\calt$, then $\Gamma$ is a uniform $(H\times T)$-lattice and hence a $\CAT(0)$ group.  Conversely, if $\Lambda$ is a uniform $(H\times T)$-lattice, then $\Lambda$ splits as a finite graph of uniform $H$-lattices with Bass--Serre tree $\calt$. 
    \item Assume $X$ is a $\CAT(0)$ polyhedral complex.  Let $\mu$ be the normalised Haar measure on $H$.  If for each local group $A_\sigma$ the kernel $K_\sigma=\Ker(\psi|_{A_\sigma})$ acts faithfully on $\calt$ and the sum $\sum_{\sigma\in VA}\Covol(A_\sigma)/|K_\sigma|$ converges, then $\Gamma$ is a $(H\times T)$-lattice.  Conversely, if $\Lambda$ is a $(H\times T)$-lattice, then $\Lambda$ splits as a graph of $H$-lattices with Bass--Serre tree $\calt$.
\end{enumerate}
\end{duplicate}

The attentive reader will have noticed in \Cref{thm.structure}\eqref{thm.gol.nonuniform} we have assumed a normalisation for the Haar measure $\mu$ on $H$.  Whilst the choice of normalisation does affect the results, the normalisation provided by Serre's covolume formula (\Cref{SerresCovolumeFormula}) makes for more pleasing statements.   Any $\Gamma$ as in \Cref{thm.structure}\eqref{thm.gol.uniform} is quasi-isometric to $X\times\calt$, see \Cref{thm.gol.qi}.  Note that Bass--Kulkarni \cite[Existence Theorem]{BassKulkarni1990} characterised the trees that admit a uniform lattice to be exactly the trees that have unimodular automorphism group.

We also introduce an analogous construction we call a \emph{complex of lattices} (Definition~\ref{def.col}) by replacing the tree with a  $\CAT(0)$ polyhedral complex.  In this setting we prove a completely analogous structure theorem (Theorem~\ref{thm.col}).  The reader may wonder why we introduce both, the answer is provided in \Cref{sec.functor} where we adapt a functor of Thomas \cite{Thomas2006}, from graphs of groups to complexes of groups, to our setting.
% In the process we deduce some consequences about commensurated subgroups of $\CAT(0)$ groups (see also Corollary~\ref{cor.commSubs2}) which generalises splittings of lattices in products of trees and twin buildings.

% \begin{corx}[Corollary~\ref{cor.commSubs}]\label{corx.commSubs}
% Let $X=X_1\times X_2$ be a proper cocompact minimal $\CAT(0)$ space and $H=\Isom(X_1)\times\Isom(X_2)$.  Suppose $X_1$ is a $\CAT(0)$ polyhedral complex.  Then, for any uniform $H$-lattice $\Gamma$, the cell stabilisers of $X_1$ in $\Gamma$ are commensurated, commensurable, and isomorphic to finite-by-$\{\Isom(X_2)$-lattices$\}$.
% \end{corx}

Let $\Gamma$ be a discrete group.  The \emph{reduced $C^\ast$-algebra} of $\Gamma$, denoted $C^\ast_r(\Gamma)$, is the norm closure of the algebra of bounded operators on $\ell^2(\Gamma)$ by the left regular representation of $\Gamma$.  We say $\Gamma$ is \emph{$C^\ast$-simple} if $C^\ast_r(\Gamma)$ has exactly two norm-closed two-sided ideals $0$ and $C^\ast_r(\Gamma)$ itself. We refer the reader to \cite{delaharpe07} for a general survey and \cite{BreuillardKalantarKennedyOzawa2017} for a number of recent developments.  In the setting of $\CAT(0)$ groups there is a characterisation of $C^\ast$-simple $\CAT(0)$ cubical groups \cite{KarSageev2016} and of linear groups \cite[Theorem~1.6]{BreuillardKalantarKennedyOzawa2017}.  In Section~\ref{sec.lat.Cstar} we study $C^\ast$-simplicity of a lattice in a product space with a tree factor (Theorem~\ref{thm.cstar.main}).  In particular we obtain a $C^\ast$-simplicity characterisation of irreducible uniform $(\Isom(\EE^n)\times T)$-lattices, see \Cref{thm.ExT.cstar}.  

A key step in the proof of \Cref{thm.ExT.cstar} is to show a uniform $(\Isom(\EE^n)\times T)$-lattice is irreducible if and only if it acts faithfully on $\calt$, see \Cref{prop.faithful.isomEn}.   Combining these results with another result of the author \cite[Theorem~A]{Hughes2022}, which gives another characterisation in terms of virtual fibring, we obtain the following characterisation of irreducible uniform $(\Isom(\EE^n)\times T)$-lattices.  Recall a group \emph{fibres} if it admits a surjection onto $\ZZ$ with finitely generated kernel.

\setcounter{thmx}{1}
\begin{thmx}
    Let $\mathcal{T}$ be a locally-finite leafless tree and $T=\Aut(\mathcal{T})$ be unimodular.  Let $n\geq1$ and $\Gamma<\Isom(\EE^n)\times T$ be a uniform lattice.  The following are equivalent:
\begin{enumerate}
    \item $\Gamma$ is an irreducible $(\Isom(\EE^n)\times T)$-lattice;
    \item $\Gamma$ acts on $\calt$ faithfully;
    \item $\Gamma$ does not virtually fibre;
    \item $\Gamma$ is $C^\ast$-simple.
\end{enumerate}
\end{thmx}

In Section~\ref{sec.biaut} we extend the work of Leary and Minasyan \cite{LearyMinasyan2019} to obtain a necessary condition for biautomaticity for lattices in a product with non-trivial Euclidean de Rham factor (Theorem~\ref{thm.notbiaut}).  We will give the relevant definitions in Section~\ref{sec.biaut}.

In Section~\ref{sec.salvetti} we adapt a construction of Horbez and Huang \cite{HorbezHuang2020} to extend actions from a regular tree to the universal cover of a Salvetti complex $\widetilde{S}_L$ with defining graph $L$.  In particular, from a graph of lattices, one obtains a complex of lattices.   We use this to construct towers of lattices in the isometry group whose covolume converges to zero.

In \cite{Thomas2006}, Thomas constructs a functor from graphs of groups covered by a fixed biregular tree $\calt$ to complexes of groups covered by a fixed ``sufficiently symmetric'' right-angled building $X$ with parameters determined by the valences of $\calt$.  We will give the relevant definitions in Section~\ref{sec.functor.rab}.  In Theorem~\ref{thm.functor.new} we show that Thomas' functor theorem takes a graph of lattices to a complex of lattices and in particular $(H\times T)$-lattices to $(H\times A)$-lattices, where $T=\Aut(\calt)$, $A=\Aut(X)$, and $H$ is a closed subgroup of the isometry group of a $\CAT(0)$ space (under mild hypothesis).  As consequences we construct more $\CAT(0)$ groups which are not virtually biautomatic (Corollary~\ref{cor.thomas.notbiaut}) and both uniform and non-uniform irreducible lattices in products of fairly arbitrary hyperbolic and Euclidean buildings (Corollary~\ref{thm.treelatstobuildings}).  We highlight two special cases here:

\begin{corx}[Special case of Corollary~\ref{cor.thomas.notbiaut}]\label{corx.thomas.notbiaut}
Let $X$ be the right-angled building of a regular $m$-gon of uniform thickness $10n$ and let $A=\Aut(X)$.  For each $n\geq 2$ there exists a weakly irreducible uniform $(\Isom(\EE^n)\times A)$-lattice which is not virtually biautomatic nor residually finite.  In particular, if $Y$ is irreducible, then the direct product of a uniform $A$-lattice with $\ZZ^2$ is not quasi-isometrically rigid even amongst $\CAT(0)$ groups.
\end{corx}

\begin{corx}[Special case of Corollary~\ref{thm.treelatstobuildings}]
    There exist irreducible uniform non-residually finite lattices in finite products of certain right-angled buildings.
\end{corx}

\subsection*{Acknowledgements}
This paper contains material from the author's PhD thesis \cite{HughesThesis}.  The author would like to thank his PhD supervisor Ian Leary for his kind natured humour, guidance, and support.  The author would like to thank Motiejus Valiunas for sharing his preprint \cite{Valiunas2021a} and Ashot Minasyan for helpful comments on an earlier draft of this paper.  Additionally, the author would like to thank Naomi Andrew, Pierre Emmanuel-Caprace, Mark Hagen, Susan Hermiller, Jingyin Huang, and Harry Petyt for helpful correspondence and conversations.  This work was supported by the Engineering and Physical Sciences Research Council grant number 2127970.  This work has received funding from the European Research Council (ERC) under the European Union’s Horizon 2020 research and innovation programme (Grant agreement No. 850930).  Finally, the author would like to thank the anonymous referee for their very careful reading of this work and their thoughtful comments

\subsection*{Bibiliographical note}
 This paper was part of a longer `paper' contained in the author's PhD thesis in a chapter of the same name (see \cite[Paper~5]{HughesThesis}).  This has been split at the request of a referee.  A sequel to this paper \emph{``Irreducible lattices fibring over the circle''} \cite{Hughes2022} will contain many of the remaining results from \cite[Paper~5]{HughesThesis} and some new results.  A short note on rational homological dimension, taken from the original version of this paper, can be found here \cite{Hughes2023hd}.  A number of group presentations, and results regarding residual finiteness, and autostackability will only exist in the thesis version.  Finally, we remark the results here have been used in \cite{Hughes2021b} to construct a lattice (and in fact a hierarchically hyperbolic group) in a product of trees which is not virtually torsion-free.  Finally, the author and M.~Valiunas used many of the results here to construct a non-biautomatic group acting geometrically on $\RH^2\times\calt_{26}$ in \cite{HughesValiunas2022,HughesValiunas2023}.

\section{Preliminaries}
In this section we recount some of the relevant background from the literature we need.  In \Cref{sec:prelims:covol}, we recount some basic definitions about lattices and Serre's covolume formula.  In \Cref{sec:prelims:nonpos}, we recount some of the work of Caprace and Monod on lattices and $\CAT(0)$ spaces.  In \Cref{sec.irr} we give an extended discussion about the definition of an irreducible lattice in a product of $\CAT(0)$ spaces.

\subsection{Lattices and covolumes}\label{sec:prelims:covol}
Let $H$ be a locally compact topological group with right invariant Haar measure $\mu$.  A discrete subgroup $\Gamma\leq H$ is a \emph{lattice} if the covolume $\mu(H/\Gamma)$ is finite.  A lattice is \emph{uniform} if $H/\Gamma$ is compact and \emph{non-uniform} otherwise.  Let $S$ be a right $H$-set such that for all $s\in S$, the stabilisers $H_s$ are compact and open, then if $\Gamma\leq H$ is discrete the stabilisers are finite.

Let $X$ be a locally finite, connected, simply connected polyhedral complex. The group $H=\Aut(X)$ of polyhedral automorphisms of $X$ naturally has the structure of a locally compact topological group, where the topology is given by uniform convergence on compacta.

\begin{thm}[Serre's covolume formula \cite{Serre1971}]\label{SerresCovolumeFormula}
Let $X$ be a locally finite simply-connected polyhedral complex.  Let $\Gamma\leq H$ be a lattice with fundamental domain $\Delta$, then there is a nomalisation of the Harr measure $\mu$ on $H$, depending only on $X$, such that for each discrete subgroup $\Gamma<H$ we have
\begin{flalign*}&&\mu(H/\Gamma)=\Vol(X/\Gamma):=\sum_{v\in \Delta^{(0)}}\frac{1}{|\Gamma_v|}. &&\qed \end{flalign*}
\end{thm}

The automorphism group $T$ of a locally finite tree $\calt$ admits uniform lattices if and only if the group $T$ is unimodular (that is the left and right Haar measures coincide) \cite[Existence Theorem]{BassKulkarni1990}.  In this case we say $\calt$ is \emph{unimodular}.  Note that further existence results for non-uniform lattices can be found in \cite{BassLubotzkyBook}.

\subsection{Non-positive curvature}\label{sec:prelims:nonpos}
We will be primarily interested in lattices in the isometry groups of $\CAT(0)$ spaces, we will call these groups \emph{$\CAT(0)$ lattices} (note that a uniform $\CAT(0)$ lattice is a $\CAT(0)$ group).  We begin by recording several facts about the structure and isometry groups of general $\CAT(0)$ spaces.  The definitions and results here are largely due to Caprace and Monod \cite{CapraceMonod2009a,CapraceMonod2009b,CapraceMonod2019}.

An isometric action of a group $H$ on a $\CAT(0)$ space $X$ is {\it minimal} if there is no non-empty $H$-invariant closed convex subset $X'\subset X$, the space $X$ is {\it minimal} if $\Isom(X)$ acts minimally on $X$.  Note that by \cite[Proposition 1.5]{CapraceMonod2009a}, if $X$ is cocompact and geodesically complete, then it is minimal. 
The {\it amenable radical} of a locally compact group $H$ is the largest amenable normal subgroup.  We can now state Caprace and Monod's group and space decomposition theorems mentioned in the introduction.

\begin{thm}\emph{\cite[Theorem 1.6]{CapraceMonod2009a}}\label{thm.CM.Isom}
Let $X$ be a proper $\CAT(0)$ space with finite dimensional Tits' boundary and assume $\Isom(X)$ has no global fixed point in $\partial X$.  There is a canonical closed, convex, $\Isom(X)$-stable subset $X'\subseteq X$ such that $G=\Isom(X')$ has a finite index, open, characteristic subgroup $H\trianglelefteq G$ that admits a canonical decomposition
\[ H\cong \Isom(\EE^n)\times S_1\times\dots\times S_p \times D_1\times\dots\times D_q,\]
for some $n,p,q\geq 0$, where each $S_i$ is an almost connected simple Lie group with trivial centre and each $D_j$ is a totally disconnected irreducible group with trivial amenable radical.\qed
\end{thm}

\begin{thm}\emph{\cite[Addendum~1.8]{CapraceMonod2009a}}
Let $X'$ and $H$ be as above, then
\[X'\cong \EE^n\times X_1\times\dots\times X_p\times Y_1\times\dots\times Y_q \]
where each $X_i$ is an irreducible symmetric space and each $Y_j$ is an irreducible minimal $\CAT(0)$-space.\qed
\end{thm}

\subsection{Irreducibility}\label{sec.irr}
Two notions of irreducibility will feature in this paper; for uniform $\CAT(0)$ lattices they are equivalent due to a theorem of Caprace--Monod \cite[Theorem~4.2]{CapraceMonod2009b}.   We remark that Caprace and Monod do not name the properties weakly and algebraically irreducible.  But as the discussion below highlights, the two properties are distinct and some care is needed in their definitions.  Indeed, the definition of weakly irreducible below is essentially extracted from the errata of \cite[Theorem~4.2]{CapraceMonod2009b}, see \cite[\S2 Theorem~4.2]{CapraceMonod2019}.

Let $X=\EE^n\times X_1\times\dots\times X_m$ be a product of irreducible proper $\CAT(0)$ spaces with each $X_i$ not quasi-isometric to $\EE^1$ and let $H= H_0\times H_1\times\dots\times H_m\coloneqq \Isom(\EE^n)\times\Isom(X_1)\times\dots\times\Isom(X_m)$, such that for each $i\geq 1$ the group $H_i$ is non-discrete, cocompact, and acting minimally on $X_i$.  Let $\Gamma$ be a lattice in $H$.

We have projections $\pi_i\colon \Gamma \to H_i$ for each factor $H_i$.  The Euclidean factor gives us two further projections, firstly, $\pi_{\Isom(\EE^n)}\colon \Gamma\to\Isom(\EE^n)\cong \RR^n\rtimes \OO(n)$ and secondly, $\pi_{\OO(n)}$ which is defined as the composition $\Gamma\to\Isom(\EE^n)\to\OO(n)$.

Suppose $n=0$, then we say $\Gamma$ is \emph{weakly irreducible} if the projection of $\Gamma$ to each proper subproduct $H_I\coloneqq \prod_{i\in I}H_i$ for each proper subset $I\subset\{1,\dots,m\}$ is non-discrete. 

Suppose $n=1$, then $\Gamma$ always virtually splits a direct product $\Gamma'\times \ZZ$ by \cite{CapraceMonod2019}. In this case we always define $\Gamma$ to be \emph{reducible}.  

Suppose $n\geq 2$.  Let $\ell$ be the maximal integer such that $\EE^n=\prod_{j=1}^\ell \EE^{k_j}$ with each $k_j\geq 1$ such that the product decomposition is preserved by \emph{some} finite index subgroup $\Lambda$ of $\Gamma$.  Observe that $\Lambda$ is contained in $\prod_{j=1}^\ell\Isom(\EE^{k_j})\times\prod_{i=1}^m H_i$.  Denote each $\Isom(\EE^{k_j})$ by $E_j$ and the corresponding orthogonal group by $O_j$.  Then for $\Gamma$ to be \emph{weakly irreducible} we require that each $k_j\geq2$, and that the projection $\pi_{I,J}$ of $\Lambda$ to each proper subproduct, $G_{I,J}\coloneqq \prod_{j\in J}E_j\times\prod_{i\in I}H_i$ for $I\subseteq\{1,\dots,m\}$ and $J\subseteq\{1,\dots,\ell\}$, of $H$ is non-discrete (here at least one of $I$ or $J$ is a proper subset).

We say $\Gamma$ is \emph{algebraically irreducible} if $\Gamma$ has no finite index subgroup splitting as the direct product of two infinite groups.

For every uniform lattice we consider in this paper the two definitions will be equivalent by \cite[Theorem~4.2]{CapraceMonod2009b}; so in this case we will simply refer to a lattice as \emph{irreducible} or \emph{reducible}.

\begin{remark} We record the following remarks concerning these definitions.
\begin{enumerate}
\item The complexity of the definition in the presence of the Euclidean factor is an artefact of the examples found by Leary--Minasyan \cite{LearyMinasyan2019}.  Here the authors constructed weakly irreducible $(\Isom(\EE^n)\times T)$-lattices (where $T$ is the automorphism group of some locally-finite tree), we will denote one of these by $\LM(A)$.  The necessity of the decomposition $\prod_j E_j$ is because $\Gamma=\LM(A)\times\LM(A)$ has non-discrete projections to any subproduct of $\Isom(\EE^{2n})\times T\times T$ but is clearly algebraically reducible.  However, $\Gamma<\Isom(\EE^{n})\times\Isom(\EE^n)\times T\times T$ and here $\Gamma$ has discrete projections to proper subproducts.  

%\item The reason we consider projections to subproducts of $H_i$ and $\OO(k_j)$ instead of the whole of $\Isom(\EE^{k_j})$ is as follows:  If $\Gamma$ projects non-discretely to some $\Isom(\EE^{k_j})$ but has a discrete projection into the corresponding $\OO(k_j)$, then $\Gamma$ virtually splits as $\Gamma'\times\ZZ$ with $\Gamma'$ infinite.  Indeed, in this case, the image of $\Gamma$ in $\OO(k_j)$ is finite, so a finite index subgroup of $\Gamma$ acts on some $\EE^{k_j}$ by translations.  In particular, this finite index subgroup fixes a point in $\partial X$, and so by the discussion after \cite[Theorem~2]{CapraceMonod2019}, $\Gamma$ splits as $\Gamma'\times\ZZ$.

\item The necessity that $k_j\geq2$ is because if some $k_j=1$, then $\pi_{\OO(n)}(\Gamma)\leq\OO(1)\cong\ZZ_2$ and so by the previous remark $\Gamma$ virtually splits as $\Gamma'\times\ZZ$.
\end{enumerate}\end{remark}

With these definitions in hand we are able to use the following result of Caprace--Monod.

\begin{thm}\emph{\cite[Theorem 4.2]{CapraceMonod2009b}}\label{thm.CMirrCrit}
Let $X$ be a proper $\CAT(0)$ space, $H<\Isom(X)$ a closed subgroup acting cocompactly on $X$, and $\Gamma<H$ a finitely generated lattice.
\begin{enumerate}
    \item If $\Gamma$ is algebraically irreducible, then for any finite index subgroup $\Gamma_0<\Gamma$ and any $\Gamma_0$-equivariant splitting $X=X_1\times X_2$ with $X_1$ and $X_2$ non-compact, the projection of $\Gamma_0$ to both $\Isom(X_1)$ and $\Isom(X_2)$ is non-discrete.
    \item If in addition the $H$-action is minimal, then the converse holds.\qed
\end{enumerate}
\end{thm}

 Finally, we restate a result of Caprace--Monod which we can use as criterion to determine non-residual finiteness of lattices in products.

\begin{thm}\emph{\cite[Theorem 4.10]{CapraceMonod2009b}}\label{thm.CM.nonresfin}
Let $X$ be a proper $\CAT(0)$ space such that $G=\Isom(X)$ acts cocompactly and minimally.  Let $\Gamma<\Isom(X)$ be a finitely generated algebraically irreducible lattice.  Let $\Gamma'=\Gamma\cap H$, where $H$ is given in Theorem~\ref{thm.CM.Isom}.  If the projection of $\Gamma'$ to an irreducible factor of $X$ has non-trivial kernel, then $\Gamma$ is not residually finite. \hfill$\qed$
\end{thm}

\section{Graphs of lattices}\label{sec.gol}

In \Cref{sec.gol:gog} we will review Bass--Serre theory and tree lattices.   In \Cref{sec:gol:structure} we define a graph of lattices and prove the structure theorem for $(H\times T)$-lattices.  In \Cref{sec:gol:reducible}, we prove a number of structural properties about reducible lattices and give some criteria for a lattice to be irreducible.  Finally, in \Cref{sec:gol:examples} we provide a couple of useful constructions and examples.  In particular, we compare our construction here to the universal covering trick of Burger and Mozes \cite[\S1.8]{BurgerMozes2000struct} and Caprace and Monod \cite[\S6.C]{CapraceMonod2009b}.  We also construct irreducible $(\Isom(\EE^n)\times T_{10n})$-lattices for all $n\geq 2$.

\subsection{Graphs of groups}\label{sec.gol:gog}
We shall state some of the definitions and results of Bass--Serre theory. In particular, the action will be on the right.  We follow the treatment of Bass \cite{Bass1993}.  Throughout a graph $A=(VA,EA,\iota,\tau)$ should be understood as it is defined by Serre \cite{SerreTrees}, with edges in oriented pairs indicated by $e$ and $\overline{e}$, and maps $\iota(e)$ and $\tau(e)$ from each edge to its initial and terminal vertices.  We will, however, often talk about the geometric realisation of a graph as a metric space.  In this case the graph should be assumed to be simplicial (possibly after subdividing) and should have exactly one undirected edge $e$ for each pair $(e,\overline{e})$.  We will often not distinguish between the combinatorial and metric notions.

A \emph{graph of groups} $(A,\cala)$ consists of a graph $A$ together with some extra data $\cala=(V\cala,E\cala,\Phi\cala)$.  This data consists of \emph{vertex groups} $A_v\in V\cala$ for each vertex $v$, \emph{edge groups} $A_e = A_{\overline{e}}\in E\cala$ for each (oriented) edge $e$, and monomorphisms $(\alpha_{e}:A_e \rightarrow A_{\iota(e)})\in\Phi$ for every oriented edge in $A$.  We will often refer to the vertex and edge groups as \emph{local groups} and the monomorphisms as \emph{structure maps}.

The \emph{path group} $\pi(\cala)$ has generators the vertex groups $A_v$ and elements $t_e$ for each edge $e\in EA$ along with the relations:
\[
    \left\{\begin{array}{c} 
        \text{The relations in the groups }A_v,\\
 t_{\overline{e}}=t_{e}^{-1},\\
 t_e\alpha_{\overline{e}}(g)t_e^{-1}=\alpha_e(g) \text{ for all } e\in EA \text{ and } g\in A_e=A_{\overline{e}}.
     \end{array}\right\}
\]

We will often abuse notation and write $\cala$ for a graph of groups.  The \emph{fundamental group of a graph of groups} can be defined in two ways.  Firstly, considering reduced loops based at a vertex $v$ in the graph of groups, in this case the fundamental group is denoted $\pi_1(\cala,v)$ (see \cite[Definition~1.15]{Bass1993}).  Secondly, with respect to a maximal or spanning tree of the graph.  Let $X$ be a spanning tree for $A$, we define $\pi_1(\cala,X)$ to be the group generated by the vertex groups $A_v$ and elements $t_e$ for each edge $e\in EA$ with the relations:
\[
    \left\{\begin{array}{c} 
    \text{The relations in the groups }A_v,\\
    t_{\overline{e}}=t_e^{-1}\text{ for each (oriented) edge }e,\\
    t_e\alpha_{\overline{e}}(g)t_e^{-1}=\alpha_{e}(g)\text{ for all }g\in A_e,\\
    t_e=1\text{ if }e\text{ is an edge in }X.
    \end{array}\right\}
\]
Note that the definitions are independent of the choice of basepoint $v$ and spanning tree $X$ and both definitions yield isomorphic groups so we can talk about \emph{the fundamental group} of $\cala$, denoted $\pi_1(\cala)$.

Let $G$ be the fundamental group corresponding to the spanning tree $X$. For every vertex $v$ and edge $e$, $A_v$ and $A_e$ can be identified with their images in $G$.  We define a tree with vertices the disjoint union of all coset spaces $G/A_v$ and edges the disjoint union of all coset spaces $G/A_e$ respectively.  We call this graph the \emph{Bass--Serre tree} of $\cala$ and note that the action of $G$ admits $X$ as a fundamental domain.

Given a group $G$ acting on a tree $\calt$, there is a \emph{quotient graph of groups} formed by taking the quotient graph from the action and assigning edge and vertex groups as the stabilisers of a representative of each orbit. Edge monomorphisms are then the inclusions, after conjugating appropriately if incompatible representatives were chosen.

\begin{thm}\emph{\cite{Bass1993}}
Up to isomorphism of the structures concerned, the processes of constructing the quotient graph of groups, and of constructing the fundamental group and Bass--Serre tree are mutually inverse.\qed
\end{thm}

Let $(A,\cala)$ and $(B,\calb)$ be graphs of groups.  A \emph{morphism of graphs of groups} $\phi:(A,\cala)\to (B,\calb)$ consists of:
\begin{enumerate}
    \item A graph morphism $f:A\to B$.
    \item Homomorphisms of local groups $\phi_v: A_v\to B_{f(v)}$ and $\phi_e=\phi_{\overline{e}}:A_e\to B_{f(e)}$.
    \item Elements $\gamma_v\in\pi_1(\calb,f(v))$ for each $v\in VA$ and $\gamma_e\in \pi(\calb)$ for each $e\in EA$ such that if $v=i(e)$ then
    \begin{itemize}
        \item $\delta_e:=\gamma_v^{-1}\gamma_e\in\calb_{f(v)}$;
        \item $\phi_a\circ\alpha_e=\Ad(\delta_e)\circ\alpha_{f(e)}\circ\phi_e$.
    \end{itemize}
\end{enumerate}

\begin{remark*}
    A specific instance of the morphism of graphs of groups will be of particular interest to us.  Let $(A,\cala)$ be a graph of groups and let $H$ be a group viewed as a graph of groups over a single vertex.  In this case a morphism $\phi\colon (A,\cala)\to H$ consists of:
    \begin{enumerate}
        \item Homomorphisms of local groups  $\phi_v: A_v\to H$ and $\phi_e=\phi_{\overline{e}}:A_e\to H$.
        \item Elements $h_v,h_e\in H$ for each $v\in VA$ and $e\in EA$ such that if $v=i(e)$, then 
        $\phi_a\circ\alpha_e=\Ad(h_v^{-1}h_e)\circ\phi_e$.
    \end{enumerate}
    Note that in this case, this is exactly the data required to define a homomorphism $\pi_1(\cala,v)\to H$.
\end{remark*}

\subsection{A Structure theorem}\label{sec:gol:structure}
In this section we will define a graph of lattices and prove the structure theorem for $(H\times T)$-lattices.

The \emph{commensurator of $A$ in $G$} is the subgroup 
\[\Comm_G(A)=\{g\in G\ |\ A\cap A^g \text{ has finite index in } A \text{ and } A^g\}.\]
We say two subgroups $A$ and $B$ of a group $G$ are \emph{commensurable in $G$} if $A\cap B$ has finite index in both $A$ and $B$.  We say $A$ and $B$ are \emph{widely commensurable in $G$} if there exist conjugates of $A$ and $B$ in $G$ that are commensurable in $G$.

\begin{defn}[Graph of lattices]\label{def.gol}
Let $H$ be a locally compact group with Haar measure $\mu$.  A \emph{graph of $H$-lattices} $(A,\cala,\psi)$ is a graph of groups $(A,\cala)$ equipped with a morphism $\psi:\cala\to H$ such that:
\begin{enumerate}
    \item For each $\sigma\in\cala$ the pair $(A_\sigma,\psi_\sigma)$ is covirtually an $H$-lattice.
    \item The local groups are commensurable in $\Gamma=\pi_1(\cala)$ and their images are commensurable in $H$.
    \item For each $e\in EA$ the element $t_e$ of the path group $\pi(\cala)$ is mapped under $\psi$ to an element of $\Comm_H(\psi_e(A_e))$.
\end{enumerate}
\end{defn}

\setcounter{thmx}{0}
\begin{thmx}[The Structure Theorem]\label{thm.structure}
Let $X$ be a proper $\CAT(0)$ space and let $H=\Isom(X)$ contain a uniform lattice.  Let $(A,\cala,\psi)$ be a graph of $H$-lattices with locally-finite Bass--Serre tree $\calt$, and fundamental group $\Gamma$.  Suppose $T=\Aut(\calt)$ is non-discrete and unimodular.
\begin{enumerate}
    \item Assume $A$ is finite.  If each local group $(A_\sigma,\psi|_{A_\sigma})$ is covirtually a uniform $H$-lattice, and the kernel $\Ker(\psi|_{A_\sigma})$ acts faithfully on $\calt$, then $\Gamma$ is a uniform $(H\times T)$-lattice and hence a $\CAT(0)$ group.  Conversely, if $\Lambda$ is a uniform $(H\times T)$-lattice, then $\Lambda$ splits as a finite graph of uniform $H$-lattices with Bass--Serre tree $\calt$. \label{thm.gol.uniform}
    \item Assume $X$ is a $\CAT(0)$ polyhedral complex.  Let $\mu$ be the normalised Haar measure on $H$.  If for each local group $A_\sigma$ the kernel $K_\sigma=\Ker(\psi|_{A_\sigma})$ acts faithfully on $\calt$ and the sum $\sum_{\sigma\in VA}\Covol(A_\sigma)/|K_\sigma|$ converges, then $\Gamma$ is a $(H\times T)$-lattice.  Conversely, if $\Lambda$ is a $(H\times T)$-lattice, then $\Lambda$ splits as a graph of $H$-lattices with Bass--Serre tree $\calt$. \label{thm.gol.nonuniform}
\end{enumerate}
\end{thmx}

We will divert the majority of the proof to the proof of Theorem~\ref{thm.col} due to the similarity of the theorem statement and arguments involved in the proof.  The difference arises from the fact that the category of graphs of groups is not equivalent to the category of $1$-complexes of groups (see \cite[Proposition~2.1]{Thomas2006}) due to a difference in morphisms.  We highlight the key differences below.

\begin{proof}
We first prove \eqref{thm.gol.uniform}.  The ``if direction" is the same as Theorem~\ref{thm.col}\eqref{thm.col.uniform}.  For the converse note that an $(H\times T)$-lattice $\Gamma$ splits as a graph of groups $(A,\cala)$ by the fundamental theorem of Bass--Serre theory and the projection to $H$ induces a morphism of graphs of groups $\pi_H:\cala\to H$.  The same argument as Theorem~\ref{thm.col}\eqref{thm.col.uniform} implies that the local groups are commensurable covirtually commensurable $H$-lattices.  In particular, the images of the elements $t_e\in\pi(\cala)$ for $e\in EA$ are contained in $\Comm_H(\pi_H(A_\sigma))$ for every local group $A_\sigma$. $\blackdiamond$

The proof of \eqref{thm.gol.nonuniform} is almost identical to \eqref{thm.gol.uniform} we will highlight the differences.  Since $X$ is a $\CAT(0)$ polyhedral complex, it follows that $X\times \calt$ is.  Now, we may apply Serre's Covolume Formula to $\Gamma=\pi_1(\cala)$.  Let $\Delta$ be a fundamental domain for $\Gamma$ acting on $X\times \calt$, then the covolume of $\Gamma$ may be computed as
\begin{align*}
\sum_{\sigma\in \Delta^{(0)}}\frac{1}{|\Gamma_\sigma|}&=\sum_{\sigma\in\pi_\calt(\Delta^{(0)})}\sum_{\tau\in\pi_\calt^{-1}(\sigma)}\frac{1}{|\Gamma_\tau|}\\
&=\sum_{\sigma\in\pi_\calt(\Delta^{(0)})}\frac{1}{|K_\sigma|}\sum_{\tau\in\pi_\calt^{-1}(\sigma)}\frac{|K_\sigma|}{|\Gamma_\tau|}\\
&=\sum_{\sigma\in\pi_\calt(\Delta^{(0)})}\frac{\Covol(\Gamma_\sigma)}{|K_\sigma|}.\end{align*}
Since $\pi_\calt(\Delta^{(0)})$ can be identified with $VA$ and the latter sum converges by assumption, it follows as before that $\Gamma$ acts faithfully properly discontinuously and isometrically with finite covolume on $X\times Y$.  For the converse we proceed as in Theorem~\ref{thm.col}\eqref{thm.col.nonuniform}. $\blackdiamond$
\end{proof}

\begin{corollary}\label{thm.gol.qi}
    Under the same hypotheses as \Cref{thm.structure}\eqref{thm.gol.uniform}, $\Gamma$ is quasi-isometric to $X\times \calt$.
\end{corollary}
\begin{proof}
    By \Cref{thm.structure}\eqref{thm.gol.uniform} $\Gamma$ acts properly discontinuously cocompactly on $X\times \calt$.  The result follows from the \v{S}varc-Milnor Lemma \cite[I.8.19]{BridsonHaefligerBook}.
\end{proof}

\subsection{Reducible lattices}\label{sec:gol:reducible}
Let $X$ be a proper minimal $\CAT(0)$ space and $H=\Isom(X)$.  Let $\mathcal{T}$ be a locally-finite leafless tree and $T=\Aut(\mathcal{T})$ be unimodular.  We will now characterise reducible uniform $(H\times T)$-lattices by both their projections to $H$ and $T$, and by the separability of the vertex stabilisers in the projection to $T$.  Moreover, if $H$ is linear, we will show that all such lattices are linear, and thus, residually finite.  We say that a subgroup $\Lambda\leq\Gamma$ is \textit{separable} if it is the intersection of finite-index subgroups of $\Gamma$, {\it virtually normal} if $\Lambda$ contains a finite index subgroup $N$ such that $N\trianglelefteq \Gamma$, and \textit{weakly separable} if it is the intersection of virtually normal subgroups of $\Gamma$.

\begin{prop}\label{thm.reducibility}
Let $X$ be a proper minimal irreducible $\CAT(0)$ space or $\EE^2$ and let $H=\Isom(X)$.  Let $\mathcal{T}$ be a locally-finite leafless tree and let $T=\Aut(\mathcal{T})$ be non-discrete and unimodular.  Let $\Gamma$ be a uniform $(H\times T)$-lattice equipped with projections $\pi_H$ and $\pi_T$ to $H$ and $T$ respectively, then the following are equivalent:
\begin{enumerate}
    \item $\pi_H(\Gamma)$ is an $H$-lattice; \label{thm.reducibility.1}
    \item $\pi_T(\Gamma)$ is a $T$-lattice; \label{thm.reducibility.2}
    \item For every vertex $v\in\mathcal{T}$, the projection of the vertex stabiliser $\pi_T(\Gamma_v)$ is separable in $\pi_T(\Gamma)$; \label{thm.reducibility.3}
    \item There is a vertex $v\in\mathcal{T}$ such that the projection of the vertex stabiliser $\pi_T(\Gamma_v)$ is weakly separable in $\pi_T(\Gamma)$; \label{thm.reducibility.4}
    \item $\Gamma$ is a reducible $(H\times T)$-lattice. \label{thm.reducibility.5}
\end{enumerate}
\end{prop}
\begin{proof}
First, we will show that \eqref{thm.reducibility.1} implies \eqref{thm.reducibility.2}, our proof for this case largely follows \cite[Proposition 1.2]{BurgerMozes2000lat}.  Assume $\pi_H(\Gamma)$ is an $H$-lattice, then $\Gamma\cdot T$ is closed and so $\Gamma\cap T$ is a uniform $T$-lattice.  Now, $\pi_T(\Gamma)$ normalises $\Gamma\cap T$ and hence by \cite[1.3.6]{BurgerMozes2000struct} is discrete.  Thus, $\pi_T(\Gamma)$ is discrete and so is a lattice in $T$.

Next, we will show that \eqref{thm.reducibility.2} implies \eqref{thm.reducibility.1}.  Assume $\pi_T(\Gamma)$ is a lattice in $T$ and consider the kernel $K$ of the action of $\Gamma$ on $\mathcal{T}$.  We will show that $K$ is a finite index subgroup of $\pi_H(\Gamma)$.  Assume that $K$ has infinite index, then $\pi_H(\Gamma)/K \leq \pi_T(\Gamma)$ is an infinite subgroup of the vertex stabiliser, a profinite group, and so cannot be discrete.  Thus, $K$ has finite index in $\pi_H(\Gamma)$.  Since $K$ acts trivially on $\mathcal{T}$ we see that $K=\Gamma\cap H$.  Since $\Gamma\cdot H$ is closed it follows $K$ is an $H$-lattice.  Thus, $\pi_H(\Gamma)$ is virtually a lattice in $H$ and therefore an $H$-lattice.

Clearly, \eqref{thm.reducibility.5} implies \eqref{thm.reducibility.1} and \eqref{thm.reducibility.2}.  We will now prove that \eqref{thm.reducibility.1} and \eqref{thm.reducibility.2} imply \eqref{thm.reducibility.5}.  By the previous paragraph we have $K\trianglelefteq\pi_H(\Gamma)$ finite index.  Let $\Gamma_T=\{\gamma \ |\ (e,\gamma)\in\Gamma\}$, we want to show that $\Gamma_T$ is a uniform $T$-lattice.  By \cite[Commensurability Theorem]{BassKulkarni1990} all uniform $T$-lattices are widely commensurable.  Thus, $\Gamma_T$ will be a finite index subgroup of $\pi_T(\Gamma)$. By the first paragraph we see $\Gamma_T$ is a uniform lattice.  Thus, $K\times\Gamma_T$ is a finite index subgroup of $\Gamma$ and so $\Gamma$ is reducible.

Now, evidently \eqref{thm.reducibility.3} implies \eqref{thm.reducibility.4}.  To see that \eqref{thm.reducibility.4} implies \eqref{thm.reducibility.5} we apply \cite[Corollary 30]{CapraceKrophollerReidWesolek2019} to $\pi_T(\Gamma)$, noting that a cocompact action on a leafless tree does not preserve any subtree, in particular, $\pi_T(\Gamma)$ is discrete.  Finally, we show that \eqref{thm.reducibility.5} implies \eqref{thm.reducibility.3}.  Observe that $\pi_T(\Gamma)$ is a virtually free $T$-lattice which splits as a finite graph of finite groups. Since $\pi_T(\Gamma)$ is a finite graph of finite groups, the vertex stabilisers are separable subgroups.
\end{proof}

One immediate consequence of the theorem is that we can determine whether a lattice is irreducible simply by considering the projections to either $H$ or $T$.  Also, note that if $H$ is non-discrete unimodular and the automorphism group of a leafless tree then we recover \cite[Proposition~1.2]{BurgerMozes2000lat} and and \cite[Corollary~32]{CapraceKrophollerReidWesolek2019}.

We also have the following observations about the linearity and residual finiteness of reducible lattices.

\begin{prop}\label{prop.red.linear}
With the same notation as \Cref{thm.reducibility}, assume $H$ is linear (or lattices in $H$ are residually finite).  If $\Gamma$ is a uniform reducible $(H\times T)$-lattice, then $\Gamma$ is linear (resp. residually finite).
\end{prop}
\begin{proof}
If $\Gamma$ is reducible, then $\Gamma$ is virtually a direct product of a linear (resp. residually finite) group with a virtually free group.  In particular, $\Gamma$ is virtually a direct product of linear (resp. residually finite) groups and therefore linear (resp. residually finite).
\end{proof}

\begin{corollary}\label{cor.trichotemy}
With the same notation as \Cref{thm.reducibility}, assume $H$ is linear.  If $\Gamma$ is a finitely generated uniform $(H\times T)$-lattice, then exactly one of the following holds:
\begin{enumerate}
    \item $\Gamma$ is reducible and therefore linear (hence residually finite);
    \item $\Gamma$ is irreducible and linear (hence residually finite);
    \item $\Gamma$ is irreducible and non-residually finite.
\end{enumerate}
Moreover, if $H$ is a connected centre-free simple linear algebraic group without compact factors and $\Gamma$ is irreducible and linear, then $\Gamma$ is arithmetic and just-infinite.
\end{corollary}
\begin{proof}
If $\Gamma$ is reducible, then $\Gamma$ is linear by \Cref{prop.red.linear}.   Now, assume $\Gamma$ is irreducible and $\pi_H(\Gamma)$ is injective, then $\pi_H$ is a faithful linear representation of $\Gamma$ and we are in the second case.  Since $\Gamma$ is linear, $\pi_T$ must be injective otherwise $\Gamma$ would contradict Theorem~\ref{thm.CM.nonresfin}.  Now, if either of $\pi_T$ or $\pi_H$ are not injective, then by Theorem~\ref{thm.CM.nonresfin} we see that $\Gamma$ is not residually finite.  Note that $\pi_T$ not being injective necessarily implies that $\pi_H$ is not injective because otherwise $\Gamma$ would admit a faithful linear representation, contradicting non-residual finiteness. To prove the `moreover' note that $\Gamma$ is just-infinite follows from the Bader--Shalom Normal Subgroup Theorem \cite{BaderShalom2006} applied to the closure of $\Gamma$ in $H\times T$.  The arithmeticity of $\Gamma$ follows from \cite{BaderFurmanSauer2019}.
\end{proof}

For a group $\Gamma$, let $b_1(\Gamma)$ denote its \emph{first Betti number}, that is, $b_1(\Gamma)=\rank_\ZZ (\Gamma^{\mathrm{ab}})$. Let $vb_1(\Gamma)$ denote the \emph{first virtual Betti number} of $\Gamma$ which is defined as
\[vb_1(\Gamma)=\sup\{b_1(\Lambda)\ |\ \Lambda \text{ has finite index in } \Gamma\} \in \ZZ_{\geq 0}\cup\{\infty\}.\]

\begin{prop}
With the same notation as \Cref{thm.reducibility}, assume $H$ is a connected centre-free semisimple linear algebraic group without compact factors.  Let $\Gamma$ be a finitely generated uniform irreducible $(H\times T)$-lattice. If $vb_1(\Gamma)>0$, then $\Gamma$ is not residually finite.  In particular, if $b_1(\calt/\Gamma)>0$, then $\Gamma$ is not residually finite.
\end{prop}
\begin{proof}
Since $\Gamma$ is irreducible, by \Cref{cor.trichotemy}, either $\Gamma$ is linear and just-infinite, or $\Gamma$ is not residually finite.  Now, if the virtual Betti number of $\Gamma$ is greater than zero, then a finite index subgroup $\Gamma'$ of $\Gamma$ admits $\ZZ$ as a quotient and so cannot be just infinite.  Hence, $\Gamma'$ is not residually finite and so neither is $\Gamma$.  

The quotient space $\calt/\Gamma$ gives rise to a graph of groups splitting of $\Gamma$ with Bass--Serre tree $\calt$.  An easy application of the $\Gamma$-equivariant Mayer-Vietoris sequence (see \cite[Chapter~IV.9]{Brown1982} for the details) applied to $\calt$ shows that $b_1(\Gamma)\geq b_1(\calt/\Gamma)$.  Alternatively note that $\Gamma\twoheadrightarrow\pi_1(\cala)\cong F_k$ where $\cala$ is the geometric realisation of the graph $\calt/\Gamma$ which is homotopy equivalent to a wedge of $k$ copies of $S^1$.
\end{proof}

\subsection{Examples}\label{sec:gol:examples}

\begin{example}[Change of tree]
Let $T_{k,\ell}$ denote the $(k,\ell)$-biregular tree.  Given an edge transitive but not vertex transitive irreducible $(H\times T_{k,\ell})$-lattice $\Gamma$ one may construct a non-residually finite irreducible $(H\times T_{mk,n\ell})$-lattice for all $m,n\geq 2$ as follows:

Firstly, note $\Gamma$ splits as a graph of $H$-lattices.  Indeed, $\Gamma=A\ast_C B$ where $A$, $B$ and $C$ are covirtually $H$-lattices.  We may assume that $A$ stabilises a vertex of valence $k$ and $B$ stabilises a vertex of valence $\ell$.  Let $N_A$ and $N_B$ be finite groups of order $m$ and $n$ respectively and pick split extensions $\widetilde{A}=N_A\rtimes A$ and $\widetilde{B}=N_B\rtimes B$.  We may construct a graph of lattices by considering the graph of groups corresponding to $\widetilde{A}\ast_C \widetilde{B}$.  The representations of $\widetilde{A}$ and $\widetilde{B}$ are the given by the composites $\widetilde{A}\twoheadrightarrow A\to H$ and $\widetilde{B}\twoheadrightarrow B\to H$.  The resulting fundamental group $\widetilde{\Gamma}$ acts on the $(mk,n\ell)$-regular tree.  Moreover, the lattice is irreducible by Proposition~\ref{thm.reducibility} and the non-discreteness of the projection to $H$.
\end{example}

We will compare the notion of a graph of lattices with the ``universal covering trick" of Burger--Mozes \cite[Section~1.8]{BurgerMozes2000struct} and generalised by Caprace--Monod \cite[Section~6.C]{CapraceMonod2009b}.  In particular, we will show how in many cases one can obtain a graph of lattices from the universal covering trick.  We take the opportunity to point out that many of the groups constructed in the previous sections cannot be obtained from universal covering trick.

\begin{example}[The universal covering trick]
Let $A$ be the geometric realisation of a locally finite graph (not reduced to a single point) and let $Q<\Isom(A)$ be a vertex transitive closed subgroup.  Let $C$ be an infinite profinite group acting level transitively on a locally finite rooted tree $\calt_0$.  Let $B$ be the $1$-skeleton of the square complex $A\times\calt_0$ and let $\calt$ be the universal cover.  Define $D$ to be the extension $1\rightarrowtail\pi_1(B)\rightarrowtail D\twoheadrightarrow C\times Q\twoheadrightarrow 1$.  By \cite[Proposition~6.8]{CapraceMonod2009b}, there exists a $\CAT(0)$ space $Y$ such that $D\rightarrowtail\Isom(Y)$ is a closed subgroup, and $D$ acts cocompactly and minimally without fixed point at infinity.

The classical situation where this is applied is as follows:  Let $Q$ be a product of $p$-adic Lie groups, $H$ be a product of real Lie groups and $\Gamma<H\times Q$ to be an $S$-arithmetic irreducible lattice.  Let $A$ be the $1$-skeleton of the Bruhat-Tit's building for $Q$, let $\calt$ be the universal cover of $A$ and let $T=\Aut(\calt)$.  Now, $\Gamma$ lifts to a weakly irreducible lattice $\widetilde\Gamma <H\times Q\times T$ and the corresponding graph of lattices is obtained by considering the graph $A/\Gamma$ equipped with local groups given by the stabilisers of the action of $\Gamma$ on $A$.
\end{example}

As an brief application we will construct (virtually) torsion-free irreducible $(\Isom(\EE^n)\times T_{10n})$-lattices.  Note that for $n$ odd these constitute the first explicit examples of such irreducible lattices.

\begin{example}\label{ex.LM.odddim.lat}
Recall from \cite[Example~9.2]{LearyMinasyan2019} the Leary-Minasyan group $\LM(A)$ where $A$ is the matrix corresponding to the Pythagorean triple $(3,4,5)$ which acts on $\EE^2\times\calt_{10}$. (Note that these groups were classified up to isomorphism by Valiunas \cite{Valiunas2020}.)  By \cite[Example~9.4]{LearyMinasyan2019}, this has presentation \[\LM(A)=\langle a,b,t\ |\ [a,b],\ ta^2b^{-1}t^{-1}=a^2b,\ tab^2t^{-1}=a^{-1}b^2\rangle.\] 
Using this group we will construct a virtually torsion-free irreducible $(\Isom(\EE^n)\times T)$-lattice where $T$ is the automorphism group of the $10n$-regular tree for all $n\geq 3$.

Let $\ZZ^n=\langle a_0,\dots,a_{n-1}\rangle$ and let $F=\langle f\rangle$ be a cyclic group of order $n$ acting on $L$ by cyclically permuting the $a_i$.  Let $L=\ZZ^n\rtimes F$, this is a crystallographic group and so embeds into $\Isom(\EE^n)$.  Now, consider the $(n\times n)$-matrix $B$ given by
\[B=\begin{bmatrix}
A & \mathbf{0}\\
\mathbf{0} & I_{n-2}
\end{bmatrix}.\]
We define $\Gamma_n$ to be the HNN extension of $L$ by the matrix $B$, the Bass--Serre tree of this HNN extension will be regular of valence $10n$.  This has generators $a_0,\dots,a_{n-1},f,t$ and relations
\[f^n=1,\ [a_i,a_j]=1,\ fa_if^{-1}=a_{i+1\pmod{n}},\ [a_2,t]=1,\ \dots,\ [a_{n-1},t]=1,\] \[ta_0^2a_1^{-1}t^{-1}=a_0^2a_1,\ ta_0a_1^2t^{-1}=a_0^{-1}a_1^2,\]
where $i,j\in\{0,\dots,n-1\}$.  Here the first three sets of relation come from $L$, the relations $[a_i,t]=1$ for $i\geq 2$ come from the fact $B$ fixes $\{a_2,\dots,a_{n-1}\}$ point-wise, and the last two relations arise from the action of $B$ on $\langle a_0,a_1\rangle$.  Now, let $a:=a_0$, then we may write $\Gamma_n$ as
\[\Gamma_n= \langle a,f,t\ |\ f^n=1,\ ta^2a^{-f}t^{-1}=a^{2}a^f,\ ta(a^2)^ft^{-1}=a^{-1}(a^2)^f,\ [a^{f^i},a^{f^j}]=1\rangle \]
for $i,j\in\{0,\dots,n-1\}$.  Thus, $\Gamma_n$ is a $3$ generator, $\frac{1}{2}n(n-1)+3$ relator group.

To see $\Gamma_n$ is irreducible note that $\pi_{\OO(n)}(\Gamma)$ is not virtually contained in some $\prod_{j}\OO(k_j)<\OO(n)$.  Indeed, consider the subgroup generated by the $\pi_{\OO(n)}(f)$-orbit of $\pi_{\OO(n)}(t)$.  To show $\Gamma_n$ is virtually torsion-free note that every torsion element of $\Gamma_n$ has non-trivial image in $\pi_{\OO(n)}(\Gamma_n)$.  This is generated by the images of $f$ and $t$ and so is a finitely generated linear group and hence has a finite index torsion-free subgroup $P_n$.  The preimage of $P_n$ in $\Gamma_n$ is torsion-free.
\end{example}

\section{An application to \texorpdfstring{$C^\ast$}{C*}-simplicity}\label{sec.lat.Cstar}

In this section we will consider the $C^\ast$-simplicity of $(H\times T)$-lattices for various locally compact groups $H$.  In \Cref{sec:Cstar:criteria} we recount the relevant definitions and derive a number of sufficient conditions for an $(H\times T)$-lattice to be $C^\ast$-simple from work of Breuillard, Kalantar, Kennedy and Ozawa \cite{BreuillardKalantarKennedyOzawa2017}.  In \Cref{sec:Cstar:app}, we give a number of applications of the criteria.  In \Cref{sec:Cstar:faithful}, we prove that a uniform $(\Isom(\EE^n)\times T)$-lattice has a faithful action on its Bass--Serre tree if and only if it is irreducible (\Cref{prop.faithful.isomEn}).  We also give an analogous statement with $T$ replaced by the automorphism group of a locally finite $\CAT(0)$ polyhedral complex (\Cref{prop.faithdul.isomEn.X}).  Finally, in \Cref{sec:faithful:Euclid} we prove that a uniform $(\Isom(\EE^n)\times T)$-lattice is $C^\ast$-simple if and only if it is irreducible.

\subsection{Some conditions for \texorpdfstring{$C^\ast$}{C*}-simplicity}\label{sec:Cstar:criteria}

Let $\Gamma$ be a discrete group.  The \emph{reduced $C^\ast$-algebra} of $\Gamma$, denoted $C^\ast_r(\Gamma)$, is the norm closure of the algebra of bounded operators on $\ell^2(\Gamma)$ by the left regular representation of $\Gamma$.  We say $\Gamma$ is \emph{$C^\ast$-simple} if $C^\ast_r(\Gamma)$ has exactly two norm-closed two-sided ideals: $0$ and $C^\ast_r(\Gamma)$ itself.

The $C^\ast$-simplicity of graphs of groups has been considered before \cite{delaHarpeJeanPhillipe2011}, however, the methods developed there are not applicable to $(H\times T)$-lattices because the vertex and edge groups are all commensurable.  Instead, we will apply the machinery developed in \cite{BreuillardKalantarKennedyOzawa2017} to prove the $C^\ast$-simplicity of $(H\times T)$-lattices via properties of either $H$ or the action on $\calt$.

Let $\Gamma$ be a group.  We say a subgroup $H$ is \emph{normalish} if for every $n\geq 1$ and $t_1\dots,t_n$ the intersection $\bigcap_{i=1}^n H^{t_i}$ is infinite.  We say $\Gamma$ has the \emph{icc property} if every non-trivial conjugacy class of $\Gamma$ is infinite.

\begin{prop}\label{Prop.pnai.gog}
Let $\Gamma$ be the fundamental group of a (possibly infinite) graph of finite groups with leafless Bass--Serre tree $\calt$ not quasi-isometric to $\RR$.  If $\Gamma$ is infinite, not virtually cyclic and acts faithfully on $\calt$, then $\Gamma$ is $C^\ast$-simple.
\end{prop}
\begin{proof}
As $\Gamma$ is not finite or virtually cyclic $\Gamma$ has a positive (possibly infinite) first $L^2$-Betti number.  Indeed, the chain complex of the Bass--Serre tree $C_\ast(\calt;\calu \Gamma)$, which is concentrated in dimension $0$ and $1$, may be used to compute the $L^2$-homology.  As $\Gamma$ is infinite the boundary map is surjective and so the $L^2$-homology is concentrated in degree $1$.  We may pair each orbit of $0$-cells $v$ with an orbit of $0$-cells $e$ contained in the boundary of $e$, in each case the dimension of the $\calu \Gamma$-module is $1/|\Gamma_v|$ or $1/|\Gamma_e|$, and $1/|\Gamma_e|-1/|\Gamma_v|\geq 0$.  Since $\Gamma$ is non-trivial and not virtually cyclic some of these inequalities must be strict.  In particular, we conclude $\Gamma$ has a (possibly infinite) non-trivial first $L^2$-Betti number.  Since $\Gamma$ has a trivial amenable radical and a non-trivial $L^2$-Betti number we may apply \cite[Theorem~6.5]{BreuillardKalantarKennedyOzawa2017} to deduce that $\Gamma$ is $C^\ast$-simple.

Alternatively, we first note that any normalish subgroup of $\Gamma$ contains a free subgroup since $\Gamma$ is a faithful graph of finite groups and is not virtually cyclic.  Now, we apply \cite[Theorem~6.2]{BreuillardKalantarKennedyOzawa2017} to deduce that $\Gamma$ is $C^\ast$-simple.
\end{proof}

\begin{thm}\label{thm.cstar.main}
Let $X=X_1\times\dots\times X_k$ be a product of proper minimal cocompact $\CAT(0)$-spaces each not isometric to $\RR$ and let $H=\Isom(X_1)\times\dots\times\Isom(X_k)$ act without fixed point at infinity.  Let $\mathcal{T}$ be a locally-finite leafless tree and $T=\Aut(\mathcal{T})$ be non-discrete unimodular.  Let $n\geq0$ and $\Gamma<\Isom(\EE^n)\times H\times T$ be a uniform lattice. Assume $\Gamma$ is weakly irreducible.  If one of the following holds: 
    \begin{enumerate}
        \item $H$-lattices have no normalish amenable subgroups;\label{thm.cstar.2a}
        \item $\Ker(\pi_T)$ is trivial and $\Ker(\pi_{\Isom(\EE^n)\times H)})$ is infinite;\label{thm.cstar.2b}
        \item $H$-lattices have a non-zero $L^2$-Betti number and trivial amenable radical;\label{thm.cstar.2c}
    \end{enumerate}
    then $\Gamma$ is $C^\ast$-simple.
\end{thm}
\begin{proof}
We will show that \eqref{thm.cstar.2a} implies $C^\ast$-simplicity.  We first show that any amenable normalish subgroup $N$ of $\Gamma$ must fix a vertex of $\calt$.  Let $g\in\Gamma$ act as a hyperbolic element on $\calt$, choose any other element $h\in\Gamma$ acting hyperbolically on $\calt$ with an axis not equal to $g$, then any normalish subgroup $N$ containing $g$ contains the free group $\langle g,g^h\rangle$ and so cannot be amenable.  Thus, $N$ fixes a vertex of $\calt$.  Now, by Theorem~\ref{thm.structure} every vertex and edge stabiliser of $\Gamma$ is a finite-by-$\{H$-lattice$\}$ group.  Since by assumption $H$-lattices do not contain any normalish amenable subgroups, neither does $\Gamma$.  It remains to verify that $\Gamma$ has no finite normal subgroups, but $\Gamma$ has trivial amenable radical by \cite[Corollary~2.7]{CapraceMonod2009b}.  In particular the result now follows from \cite[Theorem~6.2]{BreuillardKalantarKennedyOzawa2017}.

We next prove \eqref{thm.cstar.2b} implies $C^\ast$-simplicity.  Let $K=\Ker(\pi_{\Isom(\EE^n)\times H)})$, we have that $\Gamma$ is an extension of $K$ by $\pi_{\Isom(\EE^n)\times H)}(\Gamma)$.  Now, $K$ is a (possibly infinite) graph of finite groups acting faithfully on $\calt$.  Indeed, restricting $\pi:=\pi_{\Isom(\EE^n)\times H}$ to a vertex stabiliser $\Gamma_v<\Gamma$ of the action on $\calt$, by Theorem~\ref{thm.structure} we see $\Ker(\pi|_{\Gamma_v})$ is finite.  Every finite subgroup of $\Gamma$, and hence $K$, is conjugate to a finite subgroup of some vertex stabiliser.  Thus, the graph of groups decomposition is given by $\calt/K$.  

We claim $K$ is not virtually infinite cyclic.  Indeed, if $K$ was virtually cyclic, then there exists a commensurated infinite cyclic subgroup $Z<K<\Gamma$.  By \cite[Theorem 2(i)]{CapraceMonod2019} $Z$ acts properly on $\EE^n$ in the decomposition of $X$.  But $Z<K$, a contradiction.

It follows the group $K$ is $C^\ast$-simple by Proposition~\ref{Prop.pnai.gog}. Because $\Ker(\pi_T)$ is trivial, every element acts non-trivially on $\calt$ and so the centraliser $C_\Gamma(K)$ is trivial.  Now, we apply \cite[Theorem~1.4]{BreuillardKalantarKennedyOzawa2017} to prove the result.

Finally, we will prove \eqref{thm.cstar.2c} implies $C^\ast$-simplicity.  We apply the cohomology $\Gamma$-equivariant Mayer-Vietoris sequence with $\calu\Gamma$ coefficients arising from filtering $E\Gamma$ by the Bass--Serre tree \cite[Chapter~VII.9]{Brown1982}.  Since $\calt$ is not a quasi-line there is a vertex $v$ connected to an edge $e$ such that the stabilisers satisfy $|\Gamma_v:\Gamma_e|\geq 3$, thus the $L^2$-Betti numbers of $\Gamma_e$ are at least $3$ times the $L^2$-Betti number of $\Gamma_v$. Now, additivity of von Neumann dimension over exact sequences and a simple counting argument implies every $(H\times T)$-lattice must have a non-trivial $L^2$-Betti number.  Alternatively, we note that every $(H\times T)$-lattice is measure equivalent to $L\times F_r$ where $L$ is an $H$-lattice and $F_r$ is a free group.  Now, an application of the Kunneth formula yields that $L\times F_r$ has a non-trivial $L^2$-Betti number and so by Gaboriau's theorem \cite[Theorem~6.3]{Gaboriau2002} so does every $(H\times T)$-lattice.  By \cite[Corollary~2.7]{CapraceMonod2009b} every $(H\times T)$-lattice has trivial amenable radical, the result follows from \cite[Theorem~6.5]{BreuillardKalantarKennedyOzawa2017}.
\end{proof}

\subsection{\texorpdfstring{$C^\ast$}{C*}-simplicity of \texorpdfstring{$(H\times T)$}{(HxT)}-lattices}\label{sec:Cstar:app}
We now present some applications of the criteria above.

\begin{corollary}\label{cor:Cstar:Lie}
    Let $\calt$ be a locally-finite leafless tree and $T=\Aut(\calt)$ be non-discrete unimodular. Let $H$ be a finite product of simple real compact Lie groups with trivial centre.  If $\Gamma$ is a uniform $(H\times T)$-lattice, then $\Gamma$ is $C^\ast$-simple.
\end{corollary}
 Note that this appears to be new whenever $\Gamma$ is not residually finite.
\begin{proof}
   By \cite[Theorem~1]{BekkaCowleydelaHarpe1994}, we see that $H$-lattices are $C^\ast$-simple.  By \cite[Theorem~6.8]{BreuillardKalantarKennedyOzawa2017}, we see they have no normalish amenable subgroups.  In the case that $\Gamma$ is weakly irreducible, the result now follows from \Cref{thm.cstar.main}\eqref{thm.cstar.2a}.
   
   If $\Gamma$ is reducible then we proceed by induction on the number of direct factors of $H=\prod_{i=1}^kH_i$.  If $k=1$, then $\Gamma$ virtually splits as $F_n\times\Gamma_H$.  The result follows from the following three observations \cite[Proposition 19 (i,iii,iv)]{delaharpe07}, a direct product of two $C^\ast$-simple groups is $C^\ast$-simple, finite index subgroups of $C^\ast$-simple groups are simple, and a virtually $C^\ast$-simple group is $C^\ast$-simple if and only if it satisfies the icc property.  Note that the icc property follows from the fact that $\Gamma$ has no virtually normal finitely generated abelian subgroups.  Hence, $C_\Gamma(g)$ has infinite index in $\Gamma$ for all elements $g$.  Thus, the conjugacy class of $g$ is infinite by the Orbit Stabiliser Theorem.
   
   We now suppose $k\geq 1$ and note that by induction the result holds for all lattices in $H_{\hat j}\times T$ where $H_{\hat j}=\prod_{i=1, i\neq j}^kH_i$ and $1\leq j\leq k$.  Since $\Gamma$ is weakly reducible, there is some $H_{i_1},\dots,H_{i_\ell}$ such that the projection $\pi_{i_1,\dots,i_\ell}\colon\Gamma \to \prod_{j=1}^\ell H_{j_i}$ has discrete image.  Let $K_1=\prod_{j=1}^\ell H_{j_i}$ and let $K_2$ denote the remaining factors of $H\times T$.  In particular, $\Gamma$ virtually splits as a direct product of $\Gamma\cap K_1$ and $\Gamma\cap K_2$.  Lattices in both of these groups are $C^\ast$-simple by induction.  Thus, we may conclude the proof by applying \cite[Proposition 19 (i,iii,iv)]{delaharpe07} as in the case $k=1$.
\end{proof}

\begin{remark}
    Essentially the same proof as above as above applies if $H$ is the automorphism group of a direct product of affine buildings with no factors equal to $\EE^n$.  One could also prove this using \Cref{thm.cstar.main}\eqref{thm.cstar.2c}, the $\ell^2$-cohomology computations in \cite[Theorem~1.6]{PetersonSauerThom2018}, Gaboriau's Theorem \cite{Gaboriau2002}, and an argument like in \cite{Hughes2023hd}.  We do not pursue this here.
\end{remark}

We call an axial isometry $g$ of a $\CAT(0)$ space $X$ \emph{rank-one} if there is an axis $L$ for $g$ which does not bound a flat half-plane.  Here by a \emph{flat half-plane} we mean a totally geodesic embedded isometric copy of an euclidean half-plane in $X$.

\begin{corollary}
    Let $\calt$ be a locally-finite leafless tree and $T=\Aut(\calt)$ be non-discrete unimodular.   Let $X$ be a proper minimal cocompact irreducible $\CAT(0)$ polyhedral complex not isometric to $\RR$. Let $H=\Aut(X)$ and assume every uniform lattice in $H$ admits a rank-one isometry.  If $\Gamma$ is a uniform $(H\times T)$-lattice, then $\Gamma$ is $C^\ast$-simple.
\end{corollary}

Note that this result appears to be new whenever $X$ is not a cubical complex.

\begin{proof}
    First, note that since a uniform $H$-lattice $\Lambda$ admits a rank-one isometry and is not virtually cyclic, it is acylindrically hyperbolic \cite[Theorems~1.3 and 1.4(i)]{Sisto2018} (see also \cite[Corollary 5.5]{PetytSprianoZalloum2024} for a cleaner statement).  Thus, by \cite[Theorem~0.2(5)]{AbbottDahmani2019}, $\Lambda$ is $C^\ast$-simple.  If $\Gamma$ is reducible, then $\Gamma$ virtually splits as $F_n\times\Gamma_H$.  The result follows from \cite[Proposition 19 (i,iii,iv)]{delaharpe07} as in \Cref{cor:Cstar:Lie}.  
    
    We now prove that $\Lambda$ has no normalish amenable subgroups.  By \cite[Discussion before Corollary~6.6]{BreuillardKalantarKennedyOzawa2017}, we have that $\Lambda$ is in the class $\calc_{\mathrm{reg}}$ and so has $H^2_b(\Lambda;\ell^2(\Lambda))\neq 0$. It follows from \cite[Proposition~6.4]{BreuillardKalantarKennedyOzawa2017}, that $\Lambda$ has no normalish amenable subgroups.  Now, if $\Gamma$ is an irreducible $(H\times T)$-lattice, then $\Gamma$ is $C^\ast$-simple by \Cref{thm.cstar.main}\eqref{thm.cstar.2a}.
\end{proof}

\subsection{Faithful actions of \texorpdfstring{$(\Isom(\EE^n)\times T)$}{Isom(En)xT)}-lattices}\label{sec:Cstar:faithful}

The following propositions give criteria for irreducibility in terms of the action of $\Gamma$ on $\calt$.

\begin{prop}\label{prop.faithful.isomEn}
Let $\calt$ be a locally finite leafless tree not quasi-isometric to $\RR$ and let $T=\Aut(\calt)$ be non-discrete unimodular.  Let $\Gamma$ be a uniform $(\Isom(\EE^n)\times T)$-lattice.  Then $\Gamma$ is irreducible if and only if $\Gamma$ acts on $\calt$ faithfully.
\end{prop}

We will also state and prove the analogous result for lattices in the product of the automorphism group of a $\CAT(0)$ polyhedral complex (with mild hypothesis) and $\Isom(\EE^n)$.  We will prove both results simultaneously.

\begin{prop}\label{prop.faithdul.isomEn.X}
Let $X$ be an irreducible locally finite $\CAT(0)$ polyhedral complex not quasi-isometric to $\RR$ and let $A=\Aut(X)$ act cocompactly and minimally.  Let $\Gamma$ be a uniform $(\Isom(\EE^n)\times A)$-lattice.  Then $\Gamma$ is irreducible if and only if $\Gamma$ acts on $X$ faithfully.
\end{prop}

\begin{proof}[Proof of Proposition~\ref{prop.faithful.isomEn} and \ref{prop.faithdul.isomEn.X}]
Assume $\Gamma$ is irreducible.  By \cite[Corollary~3]{CapraceMonod2019}, $\Gamma$ has finite amenable radical $B$.  Such a non-trivial element $g\in B$ stabilises a vertex of the Bass--Serre tree $\calt$ (resp. complex $X$).  Now, either $g$ has infinitely many conjugates which contradicts the finiteness of $B$, or $g$ stabilises the whole of $\calt$ (resp. $X$) and so is contained in $\Gamma\cap\Isom(\EE^n)$.  

Suppose we are in the latter case, then $B$, by the graph (resp. complex) of lattices description of $\Gamma$, acts faithfully on every vertex stabiliser of $\Gamma$.  Let $\bar t$ be an infinite order element of $\pi_{\OO(n)}(\Gamma)$ such that the intersection of the complements of the fixed point subspaces of $t$ and $B$ in $\EE^n$ is positive dimensional (and so at least dimension $2$) - such an element exists by the definition of irreducibility.  Let $t$ be a lift of $\bar t$ to $\Isom(\EE^n)$.   Consider a vertex stabiliser $\Gamma_v$ in the action on $\calt$ (resp. $X$) and let $E^+$ denote a subgroup of $\Gamma_v$ corresponding to a $2$-dimensional subspace fixed by neither $t$ nor $B$. 

\noindent \textbf{Claim:} \emph{L:=$\bigcap_{n\in\ZZ}\Gamma_v^{t^n}\cap E^+$ is trivial.}

We explain how the claim implies the result:  The claim implies that every element of $E^+$ acts non-trivially on $\calt$ (resp. $X$).  Thus, since $B$ acts non-trivially on $E^+$, it must also act non-trivially on $\calt$ (resp. $X$).  A contradiction.  Thus $\Gamma$ acts faithfully on $\calt$ (resp. $X$).

\noindent \textbf{Proof of claim:}
Suppose the intersection $L$ is non-trivial, then $L$ is normalised by $\langle t\rangle$.  It follows that $\pi_{\Isom(\EE^n)}(L)$ is normalised by $
\langle t\rangle$ which is clearly nonsensical unless $L=\{1\}$.  Indeed, $\bar t$ generates a non-discrete subgroup of $\OO(n)$ and acts by conjugation with infinite order on $\pi_{\Isom(\EE^n)}(L)$.  Thus, the orbit of the $\pi_{\Isom(\EE^n)}(L)$ under conjugation by $\langle t\rangle$ in $\Isom(\EE^n)$ is non-discrete.
\hfill $\blackdiamond$

% Since $\Gamma$ is irreducible there is an infinite order element in $\pi_{O(n)}(\Gamma)$ and hence an infinite order element in $\pi_{\Isom(\EE^n)}(\Gamma)$ which does not commute with $g$.  But now the normal closure of $g$ in $\Gamma$ must contained infinitely many conjugates of $g$.  Hence, $B$ is infinite, a contradiction.  Thus, $B$ must be trivial.  

% If $\Gamma$ acts on $X$ faithfully, then the projection $\pi_A(\Gamma)$ is non-discrete.  By \cite[Proposition~3.]{Hughes2021a} it suffices to show $P=\pi_{\Isom(\EE^n)}(\Gamma)$ is non-discrete.  Suppose $P$ is discrete, then there is a finite index subgroup of $P$ isomorphic to $Z=\ZZ^n$.  But this is a virtually normal free abelian subgroup, so by \cite[Theorem~2(ii)]{CapraceMonod2019}, $\Gamma$ is reducible and so there is a finite index subgroup of $Z$ which acts trivially on $X$, a contradiction.  Thus, $P$ is non-discrete and so $\Gamma$ is weakly irreducible and by Theorem~\ref{thm.CMirrCrit} algebraically irreducible.

Suppose, for the converse, that $\Gamma$ acts on $\calt$ (resp. $X$) faithfully.  Decompose minimally the closure of $\pi_{\OO(n)}(\Gamma)$ as $\prod_{i=1}^\ell O_i$ where $O_i=\OO(k_i)$ and $\sum_{i=1}^\ell k_i=n$.  To show $\Gamma$ is irreducible we need to show that each $k_i\geq 2$, and that $\pi_T(\Gamma)$ (resp. $\pi_A(\Gamma)$), $\pi_{O_I}(\Gamma)$ and $\pi_{O_I\times T}(\Gamma)$ (resp. $\pi_{O_I\times A}(\Gamma)$) are non-discrete for each proper subset $I\subset\{1\dots,\ell\}$.

First, suppose some $k_i=1$, then by \cite{CapraceMonod2019} $\Gamma$ fixes a points in $\partial X\times \EE^n$ and so virtually splits as $\Gamma'\times\ZZ$.  It follows from \cite[Theorem~2(ii)]{CapraceMonod2019} that the direct factor $\ZZ$ acts on $X$ trivially, contradicting faithfulness of the action.  Thus, each $k_i\geq2$.

Now, by \cite[Theorem~2]{CapraceMonod2019} there is a commensurated free abelian subgroup $L$ of $\Gamma$ of rank $n$ which fixes a vertex of $\calt$ (resp. $X$).  Since $\Gamma$ acts on $\calt$ (resp. $X$) faithfully, the projection of $L$ to $T$ (resp. $A$) fixes a vertex and so is an infinite subgroup of a profinite group.  In particular, the projection of $L$ and hence, $\pi_T(\Gamma)$ (resp. $\pi_A(\Gamma)$) are non-discrete.

Next, suppose some $\pi_{O_I}(\Gamma)$ is discrete, then by an identical argument to the previous paragraph $\Gamma$ virtually splits as direct product $\Gamma'\times\ZZ$ and the direct factor $\ZZ$ acts on $X$ trivially contradicting faithfulness of the action.  Thus, the projections $\pi_{O_I}(\Gamma)$ are non-discrete for each proper subset $I\subset\{1\dots,\ell\}$.

Finally, suppose some $\pi_{O_I\times T}(\Gamma)$ (resp. $\pi_{O_I\times A}(\Gamma)$) is discrete.  Note that by \cite[Theorem~2]{CapraceMonod2019} the largest rank of a commensurated abelian subgroup of $\Gamma$ is exactly $n$.  Since $\Gamma$ acts on $\calt$ (resp. $X$) faithfully, the projection is injective.  Let $m=\sum_{i\in I}k_i$ and note that this is strictly less than $n$  It follows that the projection of $\Gamma$ is an $(\Isom(\EE^{m})\times T)$-lattice (resp. ($\Isom(\EE^{m})\times A$)-lattice).  In particular, the largest rank of a commensurated abelian subgroup of $\Gamma$ is exactly $m$, but $m<n$, a contradiction. Thus, the projections $\pi_{O_I\times T}(\Gamma)$ (resp. $\pi_{O_I\times A}(\Gamma)$) are non-discrete for each proper subset $I\subset\{1\dots,\ell\}$.
\end{proof}

\subsection{\texorpdfstring{$C^\ast$}{C*}-simplicity of \texorpdfstring{$(\Isom(\EE^n)\times T)$}{(Isom(En)xT)}-lattices}\label{sec:faithful:Euclid}

\begin{thm}\label{thm.ExT.cstar}
Let $\mathcal{T}$ be a locally-finite leafless tree and $T=\Aut(\mathcal{T})$.  Let $n\geq1$ and $\Gamma<\Isom(\EE^n)\times T$ be a uniform lattice.  Then, $\Gamma$ is $C^\ast$-simple if and only if $\Gamma$ is irreducible.
\end{thm}
\begin{proof}
Suppose $\Gamma$ is reducible, then $\Gamma$ has a finite index subgroup $\Lambda$ splitting as $\ZZ\times\Lambda'$ which has non-trivial amenable radical and so cannot be $C^\ast$-simple.  Suppose $\Gamma$ is irreducible, then by Proposition~\ref{prop.faithful.isomEn} $\Gamma$ acts on $\calt$ faithfully.  Now, $\Gamma$ is non-residually finite by Theorem~\ref{thm.CM.nonresfin} so $\Ker(\pi_{\Isom(\EE^n)})$ is non-trivial (since otherwise $\Gamma$ would be a finitely generated linear group).  Moreover, $\Ker(\pi_{\Isom(\EE^n)})$ is infinite (if it was finite it would fix $\calt$). Now, by Theorem~\ref{thm.cstar.main}\eqref{thm.cstar.2b} we deduce that $\Gamma$ is $C^\ast$-simple
\end{proof}

% \section{Autostackability}\label{sec.lat.autostackability}
% \input{5auto}

\section{Complexes of lattices}\label{sec.col}
In this section we will introduce the notion of a complex of $H$-lattices.  We will then prove a structure theorem analogous to Theorem~\ref{thm.structure} for these complexes of $H$-lattices.  

\subsection{Complexes of groups}
The definitions in this section are adapted from \cite[Section~1.4]{Thomas2006} and \cite{Haefliger1991,Haefliger1992}. Throughout this section if $X$ is a polyhedral complex then $X'$ is its first barycentric subdivision.  This is a simplicial complex with vertices $VX'$ and edges $EX'$.  Each $e\in EX'$ corresponds to cells $\tau\subset\sigma$ of $X$ and so we may orient them from $\sigma$ to $\tau$.  We will write $i(e)=\sigma$ and $t(e)=\tau$.  We say two edges $e$ and $f$ of $X'$ are \emph{composable} if $i(e)=t(f)$, in which case there exists an edge $g=ef$ of $X'$ such that $i(g)=i(e)$ and $t(g)=t(f)$, and $e,\ f$ and $g$ form the boundary of a $2$-simplex in $X'$.  We denote the set of composable edges by $E^2X'$.

A \emph{complex of groups} $G(X)=(G_\sigma,\psi_e,g_{e,f})$ over a polyhedral complex $X$ is given by the following data:
\begin{enumerate}
    \item For each vertex $\sigma$ of $VX'$, a group $G_\sigma$ called the {\it local group at $\sigma$}.
    \item For each edge $e$ of $EX'$, a monomorphism $\psi_e: G_{i(e)}\rightarrow G_{t(e)}$ called the \emph{structure map}.
    \item For each pair of composable edges $e$ and $f$, an element $g_{e,f}\in  G_{t(e)}$ called the {\it twisting element}.  We require these elements to satisfy the following conditions:
\begin{enumerate}
    \item For $(e,f)\in EX'$, we have $\Ad(g_{e,f})\psi_{ef}=\psi_e\psi_f$.
    \item For each triple of composable edges $a$, $b$ and $c$ we have a {\it cocycle condition} $\psi_{a}(g_{b,a})=g_{c,b}g_{cb,a}$.
\end{enumerate} 
\end{enumerate}
We say $G(X)$ is \emph{simple} if each of the twisting elements $g_{e,f}$ are the identity.

Some complexes of groups arise from actions on polyhedral complexes.  Let $G$ be a group acting without inversions on a polyhedral complex $Y$.  Let $X=Y/G$ with natural projection $p:Y\to X$.  For each $\sigma\in VX'$, choose a lift $\widetilde\sigma\in VY'$ such that $p\widetilde\sigma=\sigma$.  The local group $G_\sigma$ is the stabiliser of $\widetilde\sigma$ in $G$, and the structure maps and twisting elements are given by further choices.  The resulting complex of groups $G(X)$ is unique up to isomorphism.  A complex of groups isomorphic to a complex of groups arising from a group action is called \emph{developable}.

Let $G(X)$ be a complex of groups over a polyhedral complex $X$.  Let $T$ be a maximal tree in the $1$-skeleton of $X'$ and fix a basepoint $\sigma$ in $T$.  The \emph{fundamental group} of $G(X)$, denoted $\pi_1(G(X),\sigma_0)$, is generated by the set
\[\coprod_{\sigma\in VX'}G_\sigma\coprod \{e^+,e^-\colon e\in EX'\} \]
subject to the relations
\[\left\{\begin{array}{l}
\text{the relations in the groups }G_\sigma,\\
(e^+)^{-1}=e^- \text{ and } (e^-)^{-1}=e^+,\\
e^+f^+=g_{e,f}(ef)^+,\ \forall (e,f)\in E^2X',\\
\psi_e(g)=e^+ge^-,\ \forall g\in G_{i(e)},\\
e^+=1,\ \forall e\in T.
\end{array} \right\} \]

If $G(X)$ is developable, then it has a \emph{universal cover} $\widetilde{G(X)}$.  This is a simply connected polyhedral complex, equipped with an action of $G=\pi_1(G(X),\sigma_0)$ such that the complex of groups given by $\widetilde{G(X)}/G$ is isomorphic to $G(X)$.

Let $G(X)=(G_\sigma,\psi_e)$ and $H(Y)=(H_\tau,\psi_f)$ be complexes of groups over polyhedral complexes $X$ and $Y$.  Let $f:X'\to Y'$ be a simplicial map sending vertices to vertices and edges to edges.  A \emph{morphism} $\Phi:G(X)\to H(Y)$ over $f$ consists of:
\begin{enumerate}
    \item A homomorphism $\phi_\sigma:G_\sigma\to H_{f(\sigma)}$ for each $\sigma\in VX'$. 
    \item For each $e\in EX'$ an element $h_e\in H_{t(f(e))}$ such that
    \begin{enumerate}
        \item $\Ad(g_e)\psi_{f(e)}\phi_{i(e)}=\phi_{t(e)}\psi_{e}$;
        \item For all $(a,b)\in E^2X'$ we have $\phi_{t(a)}(g_{a,b})h_{ab}=h_a\psi_{f(a)}(h_b)g_{f(a),f(b)}$.
    \end{enumerate}
\end{enumerate}

\begin{remark*}
    A specific instance of the morphism of complex of groups will be of particular interest to us.  Let $G(X)=(G_\sigma,\psi_e)$ be a complex of groups and let $H$ be a group viewed as a complex of groups over a single vertex.  In this case a morphism $\Phi\colon G(X)\to H$ consists of:
    \begin{enumerate}
    \item A homomorphism $\phi_\sigma:G_\sigma\to H$ for each $\sigma\in VX'$. 
    \item For each $e\in EX'$ an element $h_e\in H$ such that
        \begin{enumerate}
            \item $\Ad(h_e)\phi_{i(e)}=\phi_{t(e)}\psi_{e}$;
            \item For all $(a,b)\in E^2X'$ we have $\phi_{t(a)}(g_{a,b})h_{ab}=h_ah_b$.
        \end{enumerate}
    \end{enumerate}
    Note that in this case, this is exactly the data required to define a homomorphism $\pi_1(G(X))\to H$.
\end{remark*}

\subsection{Complexes of lattices}
In this section we introduce complexes of lattices in analogy with the graphs of lattices we defined previously.

\begin{defn}[Complex of lattices]\label{def.col}
Let $H$ be a locally compact group with Haar measure $\mu$.  A \emph{complex of $H$-lattices} $(G(X),\psi)$ is a developable complex of groups equipped with a morphism $\psi$ to $H$ such that:
\begin{enumerate}
    \item For each $\sigma\in VX'$ the pair $(A_\sigma,\psi_\sigma)$ is covirtually an $H$-lattice.
    \item The local groups are commensurable in $\Gamma=\pi_1(G(X),\sigma)$ and their images are commensurable in $H$.
    \item For each $e\in EX'$, the elements $e^+$ and $e^-$ in $\Gamma$ are mapped to elements of $\Comm_H(\psi(G_\sigma))$.
\end{enumerate}
\end{defn}

The analogous structure theorem is given as follows.

\begin{thm}\label{thm.col}
Let $X$ be a finite dimensional proper $\CAT(0)$ space and let $H=\Isom(X)$ contain a uniform lattice.  Let $(G(Z),\psi)$ be a complex of $H$-lattices over a polyhedral complex $Z$, with universal cover $Y$, and fundamental group $\Gamma$.  Suppose $A=\Aut(Y)$ admits a uniform lattice.
\begin{enumerate}
    \item Assume $Z$ is finite and $Y$ is a $\CAT(0)$ space.  If each local group $G_\sigma$ is covirtually a uniform $H$-lattice, and the kernel $\Ker(\psi|_{G_\sigma})$ acts faithfully on $Y$, then $\Gamma$ is a uniform $(H\times A)$-lattice and hence a $\CAT(0)$ group.  Conversely, if $\Lambda$ is a uniform $(H\times A)$-lattice, then $\Lambda$ splits as a finite complex of uniform $H$-lattices with universal cover $Y$.\label{thm.col.uniform}
    \item Assume $X$ is a $\CAT(0)$ polyhedral complex and $Y$ is a $\CAT(0)$ space.  Let $\mu$ be the normalised Haar measure on $H$.  If for each local group $G_\sigma$ the kernel $K_\sigma=\Ker(\psi|_{G_\sigma})$ acts faithfully on $Y$ and the sum $\sum_{\sigma\in VZ}\Covol(G_\sigma)/|K_\sigma|$ converges, then $\Gamma$ is a $(H\times A)$-lattice.  Conversely, if $\Lambda$ is a $(H\times A)$-lattice, then $\Lambda$ splits as a finite complex of $H$-lattices with universal cover $Y$.\label{thm.col.nonuniform}
\end{enumerate}
\end{thm}

Note that by definition we are assuming all complexes of lattices are developable complexes of groups. 

\begin{proof}
We first prove \eqref{thm.col.uniform}.  The fundamental group $\Gamma$ of $G(Z)$ acts on the universal cover $Y$ and on $X$ via the homomorphism $\psi:\Gamma\to H$.  The action on the product space $X\times Y$ is properly discontinuous cocompact and by isometries.  The kernel of the action is contained in the intersection $\bigcap_{\sigma\in Z'}\Ker(\psi|_{G_\sigma})$.  But this acts faithfully on $Y$, thus, the action is faithful.  It follows $\Gamma$ is an $(H\times A)$-lattice.

We now prove the converse.  Assume $\Gamma$ is an $(H\times A)$-lattice, and note that the action of $\Gamma$ on $Y$ yields a developable complex of groups $G(Z)=(\Gamma_\sigma, \psi_a, g_{a,b})$ with spanning tree $T$ and equipped with a homomorphism $\pi_H:\Gamma\to H$.  It suffices to show the local groups corresponding to the vertices of $Z$ are covirtually $H$-lattices.  Indeed, for an edge $e\in EZ'$, if the index $|\Gamma_{t(e)}:\psi_e(\Gamma_{i(e)})|$ is infinite, then the universal cover of $G(Z)$ would not be locally finite.  It follows that all of the local groups are commensurable and hence, commensurable in $H$.  Consequently, the elements $e^+$ and $e^-$ for all $e\in E^2Z'/T$ in $\Gamma$ must commensurate the local groups.

Let $\sigma\in Y$ be a vertex and consider the stabiliser $\Gamma_\sigma<\Gamma$ for the action on $X\times Y$.  Suppose $\Gamma_\sigma$ does not act cocompactly on $X\times\sigma$, then there is no compact set whose $\Gamma_\sigma$ translates cover $X\times\sigma$.  Let $D$ be a non-compact set whose $\Gamma_\sigma$-translates cover $X\times\sigma$, but there is a compact set $C$ whose $\Gamma$ translates cover $X\times Y$.  We may arrange our subsets such that $C'=C\cap(X\times\sigma)\subseteq D$.  In particular, there are elements $g_i\in\Gamma/\Gamma_\sigma$ whose translates of $C'$ cover $D$.  But some of these elements must fix $X\times\sigma$ yielding a contradiction.  Hence, $\Gamma_\sigma$ is cocompact.

It is clear that $\Ker(\Gamma_\sigma\to H)$ is finite.  Otherwise $\Gamma$ would act with infinite point stabilisers on $X\times Y$ contradicting the discreteness of $\Gamma$.  It remains to show that the projection $\overline{\Gamma}_\sigma$ of $\Gamma_\sigma$ to $H$ is discrete.  Assume that $\overline{\Gamma}_\sigma$ is not discrete, then there does not exists a neighbourhood $N$ of $1\in H$ such that $N\cap\overline{\Gamma}_\sigma=\{1\}$.  But this implies there does not exist a neighbourbood $N'$ of $1\in H\times A$ such that $N'\cap\Gamma=\{1\}$ which contradicts the discreteness of $\Gamma$.  It follows $\Gamma_\sigma$ is covirtually an $H$-lattice.

The final step is to show the elements $e^+$ and $e^-$ for each $e\in EX'$ are mapped to elements of $\Comm_H(\pi_H(\Gamma_\sigma))$.  But this is immediate since the local groups map to $H$ with finite kernel, the elements $e^+$ and $e^-$ commensurate the local groups, and so must still preserve the appropriate conjugation relations in the map to $H$. $\blackdiamond$

The proof of \eqref{thm.col.nonuniform} is almost identical to \eqref{thm.col.uniform} we will highlight the differences.  Since $X$ is a $\CAT(0)$ polyhedral complex, it follows that $X\times Y$ is.  Now, we may apply Serre's Covolume Formula to $\Gamma$.  Let $\Delta$ be a fundamental domain for $\Gamma$ acting on $X\times Y$, then the covolume of $\Gamma$ may be computed as
\begin{align*}\sum_{\sigma\in \Delta^{(0)}}\frac{1}{|\Gamma_\sigma|}&=\sum_{\sigma\in\pi_Y(\Delta^{(0}))}\sum_{\tau\in\pi_Y^{-1}(\sigma)}\frac{1}{|\Gamma_\tau|}\\
&=\sum_{\sigma\in\pi_Y(\Delta^{(0)})}\frac{1}{|K_\sigma|}\sum_{\tau\in\pi_Y^{-1}(\sigma)}\frac{|K_\sigma|}{|\Gamma_\tau|}\\
&=\sum_{\sigma\in\pi_Y(\Delta^{(0)})}\frac{\Covol(\Gamma_\sigma)}{|K_\sigma|}. \end{align*}
Since $\pi_Y(\Delta^{(0)})$ can be identified with $Z$ and the latter sum converges by assumption, it follows as before that $\Gamma$ acts faithfully properly discontinuously and isometrically with finite covolume on $X\times Y$.  

For the converse observe that as in \eqref{thm.col.uniform} the stabilisers of $Y$ are commensurable and discrete.  If the stabilisers of $Y$ were not finite volume, then one easily obtains a contradiction using Serre's covolume formula.  The rest of the proof is identical. $\blackdiamond$
\end{proof}

\begin{corollary}\label{thm.col.qi}
    Under the same hypotheses as \Cref{thm.col}\eqref{thm.col.uniform}, $\Gamma$ is quasi-isometric to $X\times Y$.
\end{corollary}
\begin{proof}
   By \Cref{thm.col}\eqref{thm.col.uniform}, $\Gamma$ acts properly discontinuously cocompactly on $X\times Y$.  The result follows from the \v{S}varc-Milnor Lemma \cite[I.8.19]{BridsonHaefligerBook}.
\end{proof}

\section{Biautomaticity in the presence of a Euclidean factor}\label{sec.biaut}
In this section we give a condition to determine the failure of biautomaticity for a $\CAT(0)$ group in the presence of a non-trivial Euclidean de Rham factor.

For the rest of this section we fix the following notation and terminology, the treatment roughly follows \cite[Section~2]{LearyMinasyan2019} and \cite[Section~2.3, 2.5]{EpsteinEtAl}.  Let $\cala$ be a finite set and let $\Gamma$ be a group with a map $\mu:\cala\to\Gamma$.  We say that $\Gamma$ is \emph{generated by $\cala$} if the unique extension of $\mu$ to the homomorphism from the free monoid $\cala^\ast$ to $\Gamma$ is surjective.  We will call elements of $\cala^\ast$ \emph{words} and for any $w\in\cala^\ast$, if $\mu(w)=g$ for some $g\in\Gamma$, we will say \emph{$w$ represents $g$}.  We will always assume $\cala$ is closed under inversion, that is, there is an involution $i:\cala\to\cala$ such that $\mu(i(a))=\mu(a)^{-1}$, in this case we will denote $i(a)$ as $a^{-1}$.  Any \emph{subset} $\call\subseteq\cala^\ast$ will be called a \emph{language over $\cala$}.

An \emph{automatic structure} for a group $\Gamma$ is a pair $(\cala,\call)$, where $\cala$ is a finite generating set of $\Gamma$ equipped with a map $\mu:\cala\to\Gamma$ and closed under inversion, and $\cala\subseteq\cala^\ast$ is a language satisfying three conditions.  Firstly, $\mu(\call)=\Gamma$, secondly $\call$ is a \emph{regular language}, that is, it is accepted by some finite state automaton, and thirdly, it satisfies a fellow traveller property (which we will not make precise here).  We say $(\cala,\call)$ is \emph{biautomatic structure} if both $(\cala,\call)$ and $(\cala,\call^{-1})$ are automatic structures.  A group $\Gamma$ is said to be \emph{automatic (resp. biautomatic)} if it admits an automatic (resp. biautomatic) structure.

A (bi)automatic structure is \emph{finite-to-one} if $|\mu^{-1}(g)\cap\cala|<\infty$ for all $g\in\Gamma$.  As noted in \cite[Page~8]{LearyMinasyan2019} by \cite[Theorem~2.5.1]{EpsteinEtAl} it may be assumed that all (bi)automatic structures are finite-to-one. So without loss of generality we will make this assumption and we will also suppose that all the automata in this paper have no dead states.

A subgroup $H<\Gamma$ is \emph{$\call$-quasiconvex} if there exists $\kappa\geq0$ such that for any path $p$ in the Cayley graph of $\Gamma$ with respect to $\cala$, starting at $1_\Gamma$, ending at some $h\in H$, and labelled by a word $w\in\call$, then every vertex of $p$ lies in the $\kappa$-neighbourhood of $H$.  The main examples of $\call$-quasiconvex subgroups are centralisers of finite subsets as proved in \cite[Proposition~4.3]{GerstenShort1991} and \cite[Theorem~8.3.1 and Corollary~8.3.5]{EpsteinEtAl}.

\begin{thm}\label{thm.notbiaut}
Let $X=\prod_{i=1}^mX_i$ be a product of proper irreducible $\CAT(0)$ spaces each not isometric to $\EE$ and $H<\Isom(X)$ be a closed subgroup acting minimally and cocompactly on $X$.  Let $n\geq 2$ and let $\Gamma$ be an $(\Isom(\EE^n)\times H)$-lattice.  If the projection $\pi_{\Isom(\EE^n)}(\Gamma)$ is not discrete, then $\Gamma$ is not virtually biautomatic.
\end{thm}

\begin{proof}
 By \cite[Theorem~2(i)]{CapraceMonod2019} there exists a commensurated free abelian subgroup $A\leq\Gamma$ acting properly on $\EE^n$ of rank $n$.  Assume $(\calb,\call)$ is a biautomatic structure on $\Gamma$.

\noindent\textbf{Claim:} \textit{There is a finite index subgroup $B$ of $A$ that is the centraliser of a finite set $\cals\subseteq G$.}%  Moreover, $B$ is $\call$-quasiconvex.}

\noindent\textbf{Proof of Claim:} By the Flat Torus Theorem the rank of a maximal abelian subgroup of $\Gamma$ is bounded by the rank of a maximal flat in $X\times\EE^n$.  Let $F$ be such a flat acted on by $A$.  Fix a set of generators $S_A$ for $A$ and a set of generators $S$ containing $S_A$ for the maximal abelian subgroup containing $A$ stabilising $F$.

We may split $X$ into a product $Y_1\times Y_2$ where $A$ acts trivially on $Y_1$ and non-trivially on $Y_2$.  For $j=1,2$ let $K_j=\Isom(Y_j)\cap H$.  Since, $A$ acts trivially on $Y_1$ it follows $A$ and $\Gamma\cap K_1$ commute.  Now, $\Gamma$ splits as a complex of $(\Isom(\EE^n)\times K_1)$-lattices.  Note that any symmetric space factor of $X$ is contained in $Y_1$ since $A$ projects trivially by \cite[Lemma~6]{CapraceMonod2019}.  In particular, $A$ is a subgroup of a vertex group $\Gamma_v$, which is covirtually virtually isomorphic to $A\times K_v$, where $K_v$ is a lattice in $K_1$. Define $S_K$ to be a set of generators for $K_v$ and for each $s\in S_K$ let $s'\in K_v$ be some element which does no commute with $s$.  Define a set $S_K'=\{s,s'\colon s\in S_K\}$ and note that it is finite.

Let $N=\Ker(\pi_{\Isom(\EE^n)})$.  For each irreducible factor $Z_j$ for $j=1,\dots,\ell$ of $Y_2$ choose some element $g_j\in N <\Gamma$  which acts non-trivially on $Z_j$.  Note the kernel $N$ is non-empty since otherwise $\Gamma$ would be a finitely generated linear group and hence residually finite, contradicting \cite[Theorem~2(iv)]{CapraceMonod2009a}.  Now, we can choose such an element so that it centralises a finite index subgroup of $A$.  Indeed, we may choose $g_j\in N$.  Since $A$ is commensurated and $N$ is a normal subgroup, $g_j$ centralises a finite index subgroup of $A$.  For each $g_j$ pick another element $g_j'$ which centralises a finite index subgroup of $A$ and does not commute with $g_j$.  Let $S_{Y_2}=\{g_j,g_j'\colon j=1\dots,\ell\}$ and note that it is finite. Let $B=\left(\bigcap_{g\in S_{Y_2}}A^g\right)\cap A$.  Since $B$ is the intersection of $A$ with finitely many subgroups of $\Gamma$ commensurable with $A$, we see $B$ is a finite index subgroup of $A$.  By construction $B\coloneqq A'$ is the centraliser of the finite set $\cals \coloneqq  S_K'\cup S_{Y_2}\cup S_A$.  %Thus, by \cite[Proposition~4.3]{GerstenShort1991}, $A'$ is $\call$-quasiconvex. 
\hfill $\blackdiamond$

 We aim to apply a result of Leary and Minasyan \cite[Corollary~5.4]{LearyMinasyan2019}, which states: 

\emph{Let $G$ be a biautomatic group and let $X\subseteq G$ be a finite subset such that $B\coloneq C_G(X)$ is abelian.  Then, there is a finite index subgroup $\Comm_G^0(B)$ of $\Comm_G(B)$ such that every finitely generated subgroup of $\Comm_G^0(H)$ centralises a finite-index subgroup of $B$ in $G$.} 

In our case we take our $B$ to be in the claim, noting it is the centraliser of the finite set $\cals$.  Now, since $B$ is a commensurated subgroup of $\Gamma$, by the aforementioned result of Leary--Minasyan, there is a finite index subgroup $\Gamma^0\trianglelefteq\Gamma$ such that every finitely generated subgroup of $\Gamma^0$ centralises a finite index subgroup of $B$.

As $\pi_{\Isom(\EE^n)}(\Gamma)$ is not discrete, there exists an element in $\overline{t}\in\pi_{O(n)}(\Gamma)$ with infinite order, let $t$ denote a preimage of $\overline{t}$ in $\Gamma$.   After passing to a suitable power we may assume $t^k\in\Gamma^0$.  But $\langle t^k\rangle$ does not centralise any finite index subgroup of $A$ (hence also of $B$), a contradiction.  Thus, there is no biautomatic structure on $\Gamma$.  Since the hypotheses on $\Gamma$ pass to finite index subgroups, it follows $\Gamma$ is not virtually biautomatic.
\end{proof}

\begin{example}
The group $\Gamma_n$ for each $n\ge2$ constructed in Example~\ref{ex.LM.odddim.lat} is an irreducible $(\Isom(\EE^n)\times T_{10n})$-lattice that is not virtually biautomatic.
\end{example}

\begin{remark}
In light of M.~Valiunas' result \cite[Theorem~1.2]{Valiunas2021a} Theorem~\ref{thm.notbiaut} can be strengthened to state that $\Gamma$ does not embed into any biautomatic group.  It may also be possible to simplify the proof using their result.
\end{remark}

\section{Products with Salvetti complexes}\label{sec.salvetti}
In this section we will adapt a construction of Horbez and Huang \cite[Proposition~4.5]{HorbezHuang2020} to extend actions from trees to Salvetti complexes.  Horbez--Huang constructed an example of a non-uniform lattice acting on the universal cover of the Salvetti complex $\widetilde X_L$ provided $L$ is not a complete graph.  We adapt this to construct a tower of uniform lattices in $\Aut(\widetilde X_L)$.

\subsection{Right-angled Artin groups and Salvetti complexes}

Let $L$ be a flag complex.  We begin by with a wedge of circles (each made of a single vertex and edge) attached along a common vertex $x$.  For each edge $\{v,w\}$ in $L$ we attach a $2$-torus along the word $vwv^{-1}w^{-1}$.  Continuing inductively, for each $n\geq 2$ cell $\{v_1,\dots,v_n\}$ in $L$ we attach an $n$-torus such that the faces correspond to already attached $(n-1)$-tori.  We denote the resulting space by $X_L$ and call it the \emph{Salvetti complex} of $L$.

The fundamental group $A_L\coloneqq \pi_1(X_L)$ is the \emph{right-angled Artin group} (RAAG) on $L$.  The group has a generating set given by the vertices of $L$ and the relations that two generators $v$ and $w$ commute if and only if they are joined by an edge in $L$.  We denote the universal cover of the Salvetti complex $X_L$ by $\widetilde{X}_L$ and the isometry group of $\widetilde{X}_L$ by $S_L$.

\subsection{Extending actions over Salvetti complexes}
We will now adapt the construction of Horbez and Huang \cite[Proposition~4.5]{HorbezHuang2020} to extend actions from trees to Salvetti complexes and present some applications.  Let $\calt_{2k}$ denote the $2k$-regular tree and let $T_{2k}$ denote $\Aut(\calt_{2k})$.

\begin{construction}\label{construction.salvetti}
\emph{Let $L$ be a finite simplicial graph on vertices $\{v_1,\dots,v_m\}$ and suppose $\langle v_1,\dots, v_k\rangle=F_{k}<A_L$ is a free subgroup.  Let $\Gamma$ be a group acting on $\calt_{2k}$ by isometries such that the action is label-preserving, then the action of $\Gamma$ on $\calt_{2k}$ extends to an action of $\widetilde{\Gamma}$ on $\widetilde{S}_L$ by isometries.  Moreover, if $\Gamma$ is a $T_{2k}$-lattice then $\widetilde{\Gamma}$ is an $\Aut(\widetilde{S}_L)$-lattice.}
\end{construction}
\begin{proof}
Define $\calv=\{v_1,\dots,v_k\}$.  Define $\phi:A_L\twoheadrightarrow F_{k}$ by $v\mapsto 1$ unless $v\in\calv$ and let $\pi:\widetilde{S}_L\to X$ be the covering space corresponding to $\Ker(\phi)$.  Let $\Gamma$ be a group acting on $\calt_{2k}$ preserving the labelling, we want to extend the action of $\Gamma$ on $\calt_{2k}$ to an action on $X$.

We may identify the vertex set of $\calt_{2k}$ with the vertex set of $X$ via the embedding of $\calt_{2k}\rightarrowtail X$.  We orient each edge of $\widetilde{S}_L$ and endow $X$ with the induced labelling and orientation.  The $1$-skeleton $X^{(1)}$ of $X$ is obtained from $\calt_{2k}$ by attaching to each vertex of $\calt_{2k}$ a circle for each $v\in VL\backslash \calv$.

Since $\Gamma$ acts by isometries on $\calt_{2k}$ label preservingly, it follows $\Gamma$ acts by isometries on $X^{(1)}$ label preservingly and preserves the orientation of edges in $VL\backslash \calv$.  It follows the action extends to $X$.  Let $\widetilde{\Gamma}$ be the group of lifts of all automorphisms in $\Gamma$, we have a short exact sequence
\[\begin{tikzcd}
1 \arrow[r] & \Aut(\pi) \arrow[r] & \widetilde{\Gamma} \arrow[r] & \Gamma \arrow[r] & 1.
\end{tikzcd} \]

We have $\widetilde{S}_L/\widetilde{\Gamma}=X/\Gamma$ so there is a bijection between the $\widetilde{\Gamma}$-orbits of $\widetilde{S}_L^{(0)}$ and the $\Gamma$-orbits of $\calt_{2k}^{(0)}$.  For a vertex $v\in X$, each lift of $g\in\rm{Stab}_\Gamma(v)$ fixes a unique vertex $\tilde{v}\in\widetilde{S}_L$.  In particular, the cardinality of the vertex stabilisers is preserved.  It follows from Serre's covolume formula that if $\Gamma$ was a $T_{2k}$-lattice, then $\widetilde{\Gamma}$ is an $\Aut(\widetilde{S}_L)$-lattice.
\end{proof}

A sequence of lattices $(\Gamma_n)_{n\geq 1}$ in a locally compact group $H$ is \emph{an ascending tower} if we have proper inclusions $\Gamma_n\rightarrowtail \Gamma_{n+1}$.  Note that if a sequence of lattices $(\Gamma_n)_{n\geq 1}$ is an ascending tower, then the covolumes $\Covol(\Gamma_n)$ converge to $0$ as $n$ goes to infinity.

\begin{prop}
 There is an ascending tower of lattices in $T_{4}=\Aut(\calt_{4})$ with label-preserving action.
\end{prop}
\begin{proof}
Let $r\geq 2$ and $s\geq1$. By \cite[Proposition~7.10]{BassKulkarni1990} there exists, in $T_3$, a tower of uniform lattices $\Gamma_{r^s}$, with $r$ fixed and $s\to \infty$, and with fundamental domain a path of length $2$ wedged onto the basepoint of a single edge loop.  By \cite[Lemma~9.2]{Huang2018} (see also \cite[Proposition~4.1]{Hughes2022} for a more explicit statement) we obtain a sequence of lattices $\Gamma_{r^s}'$ acting label preservingly on $\calt_4$.  To see that $\Gamma_{r^s}\rightarrowtail \Gamma_{r^{s'}}$ induces $\Gamma_{r^s}'\rightarrowtail \Gamma_{r^{s'}}'$ note that by \cite[Proposition~4.1]{Hughes2022} all of the groups admit a short exact sequence \[1\to F_\infty \to \Gamma_{r^s}'\to \Gamma_{r^s}\to 1 \]
and that the extensions are all compatible because the fundamental domain does not depend on $r$ or $s$.
Passing to the orientation preserving, index at most $2$, subgroup we obtain a tower of lattices $\Lambda_{r^s}$.  To see the covolume converges to $0$ note that the fundamental domain of $\Lambda_{r^s}$ is fixed for all $s$ but the stabilisers increase in order.  The result now follows from Serre's covolume formula.
% The groups will be index two subgroups of the HNN extensions constructed in \cite[Example~7.4]{BassKulkarni1990}.  We describe them here for the convenience of the reader.  Let $V_r=\{f:\ZZ_r\to\ZZ_2\ \colon\ f\text{ a function}\}\cong\ZZ_2^r$ and $\alpha_r\in\Aut(V_r)$ by $\alpha_r(f)(i)=f(i+1)$.  Let $W_r=\{f\in V_r\ \colon\ f(0)=1\}\cong\ZZ_2^{r-1}$ and define $\Gamma_r$ to be the HNN extension
% \[\langle V_r,t\ |\ f^t=\alpha_r(f)\ \forall f\in W_r\rangle. \]
% By \cite[Proposition~7.6]{BassKulkarni1990} the group $\Gamma_r$ acts faithfully on $\calt_{4}$ with quotient a loop (one vertex and one edge) and covolume $1/m^r$.  Moreover, if $r|r'$ then $\Gamma_r\leq\Gamma_{r'}$ with index $m^{r'-r}$ and so for $r\geq2$, the sequence $(\Gamma_{r^s})_{s\geq1}$ is an infinite ascending chain in $\rm{Lat}_u(\calt_4)$.
%
% Now, define $\phi:\Gamma_r\to\ZZ_2$ by $\phi(V_r)=0$ and $\phi(t)=1$.  The kernel $\Lambda_r$ is an index two subgroup which satisfies the same properties as $\Gamma_r$ except now the quotient has fundamental domain the first barycentric subdivision of a loop (two vertices and two edges) and covolume $2/m^r$.
\end{proof}

\begin{corollary}
Let $L$ be a finite flag complex which is not a full simplex.   Then the automorphism group of the universal cover of the Salvetti complex contains an ascending tower of uniform lattices $(\Gamma_n)_{n\geq 1}$.  In particular, $\Covol(\Gamma_n)$ converges to $0$ as $n$ goes to infinity. 
\end{corollary}
\begin{proof}
Fix $r\geq2$.  We apply Construction~\ref{construction.salvetti} to the lattices $\Lambda_{r^s}$ for $s\geq1$ in the preceding proposition and obtain a sequence of lattices $\widetilde{\Lambda}_{r^s}$ in $S_L$.  To see the covolume converges to $0$ note that the fundamental domain of $\Lambda_{r^s}$ is fixed for all $s$ but the stabilisers increase in order, then apply Serre's covolume formula. It remains to show that the inclusions $\Lambda_{r^s}\rightarrowtail\Lambda_{r^{s'}}$ induce inclusions $\widetilde{\Lambda}_{r^s}\rightarrowtail\widetilde{\Lambda}_{r^{s'}}$ for $s'<s$.  Consider the covering space $\pi:\widetilde{X}_L\to X$ where $X$ is as in Construction~\ref{construction.salvetti}.  Note that $X$ and hence $\Aut(\pi)$ does not depend on $r$ or $s$ since each group acts with the same fundamental domain.  In particular, as $\Lambda_{r^s}<\Lambda_{r^{s'}}$ we have $\widetilde{\Lambda}_{r^s}<\widetilde{\Lambda}_{r^{s'}}$ for $s< s'$.  We set $\Gamma_n=\widetilde \Lambda_{r^n}$ to complete the proof.
\end{proof}

% Let $J$ be a simplicial complex on $[m]$ and let $V\subseteq [m]$.  The \emph{double of $J$ over $V$}, denoted $\cald(J;V)$ is the simplicial complex with vertices $[m]\backslash V \cup \{v^+,v^-\colon v\in V\}$ and simplices described as follows: $[w_1^{\epsilon_1},\dots,w_n^{\epsilon_n}]$, where $\epsilon_i\in\{+,-,\ \ \}$, spans an $n$-simplex in $\cald(J,V)$ if and only if $[w_1,\dots,w_n]$ spans an $n$-simplex in $J$.

\begin{thm}\label{thm.salvetti.ext}
Let $L$ be a flag complex on $m$ vertices.  Let $X$ be a proper $\CAT(0)$ space and assume $H<\Isom(X)$ acts cocompactly and minimally.  The following conclusions hold:
\begin{enumerate}
    \item Let $\Gamma$ be a group acting on $\calt_{2k}$ by label-preserving isometries.  Then the action of $\Gamma$ on $\calt$ extends to an action of $\widetilde{\Gamma}$ on $\widetilde{S}_L$ by isometries.\label{thm.salvetti.ext.1}
    \item If $\Gamma$ is a uniform lattice in $H\times T_{2k}$, then $\widetilde{\Gamma}$ is a uniform lattice in $H\times\Aut(\widetilde{S}_L)$.\label{thm.salvetti.ext.2}
    \item If in addition $X$ is a $\CAT(0)$ polyhedral complex and $\Gamma$ is an $(H\times T_{2k})$-lattice, then $\widetilde{\Gamma}$ is an $(H\times\Aut(\widetilde{S}_L))$-lattice.\label{thm.salvetti.ext.3}
    \item If the projection of $\Gamma$ to $H$ (resp. $T_{2k}$) is non-discrete, then so is the projection of $\widetilde{\Gamma}$ to $H$ (resp. $\Aut(\widetilde{S}_L)$).\label{thm.salvetti.ext.4}
\end{enumerate}
\end{thm}
\begin{proof}
The proof of \eqref{thm.salvetti.ext.1} is exactly Construction~\ref{construction.salvetti}.

Now, let $\Gamma$ be the group as in the theorem statement.  The desired group $\widetilde{\Gamma}$ can now be described as the group of lifts of all automorphisms in $\Gamma$. In particular we obtain an extension $\Aut(\pi)\to\widetilde{\Gamma}\to\Gamma$. $\blackdiamond$

The proof of \eqref{thm.salvetti.ext.2} follows from taking the diagonal embedding $\widetilde{\Gamma}\rightarrowtail H\times\Aut(\widetilde{S}_L)$ and then noting that the quotient $(\widetilde{S}_L\times X) / \widetilde{\Gamma}$ is compact and that cardinality of each of the vertex stabilisers is finite. $\blackdiamond$ 

We prove \eqref{thm.salvetti.ext.3} in the same manner, noting the covolume on the product space is finite by Serre's Covolume Formula. $\blackdiamond$

The images of the projections of $\Gamma$ and $\widetilde{\Gamma}$ to $H$ coincide.  Since any element of $\Gamma$ which acts non-trivially on $\calt_{2k}$ lifts to an element acting non-trivially on $\widetilde{S}_L$, the non-discreteness of $\pi_{T_{2k}}(\Gamma)$ implies the non-discreteness of $\pi_{\Aut(\widetilde{S}_L)}(\widetilde{\Gamma})$.  This proves \eqref{thm.salvetti.ext.4}. $\blackdiamond$
\end{proof}

\section{A functor theorem}\label{sec.functor}
In this section will show that the functors introduced by A. Thomas in \cite{Thomas2006} take graphs of $H$-lattices with a fixed Bass--Serre tree to complexes of $H$-lattices whose development is a ``sufficiently symmetric" right-angled building (we will make this precise later).  Finally, we will combine these tools to construct a number of examples.  In particular, non-residually finite $(\Isom(\EE^n)\times A)$-lattices where $A$ is the automorphism group of a sufficiently symmetric right-angled building, and non-residually finite algebraically irreducible lattices in products of arbitrarily many isometric and non-isometric sufficiently symmetric right-angled buildings.

\subsection{Right angled buildings} \label{sec.functor.rab}
We now recount the background we need, fore more information the reader is referred to Davis' book \cite{Davis2008}.  Let $(W,I)$ be a right-angled Coxeter system.  Let $N$ be the finite nerve of $(W,I)$ and $P'$ be the simplicial cone on $N'$ with vertex $x_0$.  A \emph{right-angled building} of type $(W,I)$ is a polyhedral complex $X$ equipped with a maximal family of subcomplexes called \emph{apartments}.  Such an apartment is isometric to the Davis complex for $(W,I)$ and the copies of $P'$ in $X$ are called \emph{chambers}.  Moreover, the apartments and chambers satisfy the axioms for a Bruhat--Tits building.

Let $\cals$ denote the set of $J\subseteq I$ such that $W_J\leq W$ is finite.  Note that $W_\emptyset=\{1\}$ so $\emptyset\in\cals$. For each $i\in I$, the vertex $P'$ of \emph{type} $\{i\}$ will be called an \emph{$i$-vertex}, and the union of the simplices of $P'$ which contains the $i$-vertex but not $x_0$ will be called the \emph{$i$-face}  There is a one-to-one correspondence between the vertices of $P'$ and the types $J\in\cals$.

Let $X$ be a right-angled building.  A vertex of $X$ has a type $J\in\cals$ induced by the types of $P'$.  For $i\in I$ an \emph{$\{i\}$-residue} of $X$ is the connected subcomplex consisting of all chambers which meet in a given $i$-face.  The cardinality of the $\{i\}$-residue is the number of copies of $P'$ in it.

\begin{thm}[\cite{HaglundPaulin2003}]
Let $(W,I)$ be a right-angled Coxeter system and $\{q_i\colon i\in I\}$ a set of integers such that $q_i\geq 2$.  Then up to isometry there exists a unique building $X$ of type $(W,I)$ such that for each $i\in I$ the $\{i\}$-residue of $X$ has cardinality $q_i$.
\end{thm}

If $(W,I)$ is generated by reflections in an $n$-dimensional right-angled hyperbolic polygon $P$, then $P'$ is the barycentric subdivision of $P$.  Moreover, the apartments of $X$ are isometric to $\RH^n$.  In this case we call $X$ a \emph{hyperbolic building}.  %We remark that a right-angled building can be expressed as the universal cover of a polyhedral product, however, we will not use this observation elsewhere.

% \begin{remark}
% Let $(W,I)$ be a right-angled Coxeter system with parameters $\{q_i\}$ and nerve $N$.  Let $E_i$ be a set of size $q_i$ and let $CE_i$ denote the simplicial cone on $E_i$, denote the collections of these by $\underline{E}$ and $C\underline{E}$ respectively.  The right-angled building of type $(W,I)$ with parameters $\{q_i\}$ is the universal cover of the polyhedral product $(C\underline{E},\underline{E})^N$.
% \end{remark}

\subsection{A functor theorem}
In this section we will recap a functorial construction of A.~Thomas which takes graphs of groups with a given universal covering tree to complexes of groups with development a right-angled building.  We will then show that this functor takes graphs of lattices to complexes of lattices and deduce some consequences.

Let $\calt$ be a tree, let $X$ be a right-angled building.  We define some categories: 
\begin{itemize}
    \item Let $\calg$ denote the category with objects graphs of groups and with arrows given by morphisms of graphs of groups.
    \item Let $\calg(\calt)$ denote the subcategory of $\calg$ with objects consisting of graphs of groups with universal covering tree $\calt$.
    \item Let $\calc$ denote the category with objects complexes of groups and with arrows given by  morphisms of complexes of groups
    \item Let $\calc_1$ denote the subcategory of $\calc$ consisting of complexes of groups over 1–dimensional polyhedral complexes (that is, simplicial graphs), and morphisms over nondegenerate polyhedral maps.
    \item Let $\calc(X)$ denote the subcategory of $\calc$ with objects consisting of developable complexes of groups with unviversal cover $X$.
\end{itemize}

\begin{defn}
Let $X$ be a right-angled building of type $(W,I)$ and parameters $\{q_i\}$ with chamber $P'$.  Suppose $m_{i_1,i_2}=\infty$ and define the following two symmetry conditions due to Thomas \cite{Thomas2006}:
\begin{enumerate}[label=(T\arabic*)]
    \item There exists a bijection $g$ on $I$ such that $m_{i,j}=m_{g(i),g(j)}$ for all $i,j\in I$, and $g(i_1)=i_2.$ \label{Thomas.1}
    \item There exists a bijection $h:\{i\in I: m_{i_1,i}<\infty\}\to \{i\in I:m_{i_2,i}<\infty\}$ such that $m_{i,j}=m_{h(i),h(j)}$ for all $i,j$ in the domain, $h(i_1)=i_2$, and for all $i$ in the domain $q_i=q_{h(i)}$. \label{Thomas.2}
\end{enumerate}
\end{defn}

We include the construction adapted from \cite{Thomas2006} for completeness and for utility in the proofs of the new results which will follow.  An example of the construction for a graph of groups consisting of a single edge is given in Figure~\ref{fig.pentagon}

\newdimen\R
\R=2.3cm
\begin{figure}[h!]
    \centering
    \begin{tikzpicture}[scale=0.7]
 \draw[xshift=0.0\R] (0:\R) \foreach \x in {72,144,...,359} {
            -- (\x:\R) } 
        -- cycle (360:\R) node[circle,fill=black,inner sep=0pt,minimum size=5pt,label=right:{$\{i_2,i_4\}$}] (c) {}
        -- cycle (288:\R) node[circle,fill=black,inner sep=0pt,minimum size=5pt,label=below right:{$\{i_2,i_5\}$}] (d) {}
        -- cycle (216:\R) node[circle,fill=black,inner sep=0pt,minimum size=5pt,label=below left:{$\{i_1,i_5\}$}] (e) {}
        -- cycle (144:\R) node[circle,fill=black,inner sep=0pt,minimum size=5pt,label=above left:{$\{i_1,i_3\}$}] (a) {}
        -- cycle  (72:\R) node[circle,fill=black,inner sep=0pt,minimum size=5pt,label=above right:{$\{i_3,i_4\}$}] (b) {};
        \draw (a) -- (b) node [midway, circle,fill=black,inner sep=0pt,minimum size=5pt,label=above left:{$\{i_3\}$}] (ab) {};
        \draw (b) -- (c) node [midway, circle,fill=black,inner sep=0pt,minimum size=5pt,label=above right:{$\{i_4\}$}] (bc) {};
        \draw (c) -- (d) node [midway, circle,fill=black,inner sep=0pt,minimum size=5pt,label=below right:{$\{i_2\}$}] (cd) {};
        \draw (d) -- (e) node [midway, circle,fill=black,inner sep=0pt,minimum size=5pt,label=below:{$\{i_5\}$}] (de) {};
        \draw (e) -- (a) node [midway, circle,fill=black,inner sep=0pt,minimum size=5pt,label=left:{$\{i_1\}$}] (ea) {};
        \node [circle,fill=black,inner sep=0pt,minimum size=5pt,label={[right,xshift=0.34cm,yshift=0.125cm]:{$\emptyset$}}] (m) {};
        \draw (m) -- (a);
        \draw (m) -- (b);
        \draw (m) -- (c);
        \draw (m) -- (d);
        \draw (m) -- (e);
        \draw (m) -- (ab);
        \draw (m) -- (bc);
        \draw[dashed] (m) -- (cd);
        \draw (m) -- (de);
        \draw[dashed] (m) -- (ea);
\end{tikzpicture}  
\begin{tikzpicture}[scale=0.7]
 \draw[xshift=0.0\R] (0:\R) \foreach \x in {72,144,...,359} {
            -- (\x:\R) } 
        -- cycle (360:\R) node[circle,fill=black,inner sep=0pt,minimum size=5pt,label=right:{$G_w\times \ZZ_{q_4}$}] (c) {}
        -- cycle (288:\R) node[circle,fill=black,inner sep=0pt,minimum size=5pt,label=below right:{$G_w\times \ZZ_{q_5}$}] (d) {}
        -- cycle (216:\R) node[circle,fill=black,inner sep=0pt,minimum size=5pt,label=below left:{$G_v\times \ZZ_{q_5}$}] (e) {}
        -- cycle (144:\R) node[circle,fill=black,inner sep=0pt,minimum size=5pt,label=above left:{$G_v\times\ZZ_{q_3}$}] (a) {}
        -- cycle  (72:\R) node[circle,fill=black,inner sep=0pt,minimum size=5pt,label={[right,xshift=0cm]:{$G_e\times \ZZ_{q_3}\times\ZZ_{q_4}$}}] (b) {};
        \draw (a) -- (b) node [midway, circle,fill=black,inner sep=0pt,minimum size=5pt,label={[above,xshift=-0.15cm]:{$G_e\times\ZZ_{q_3}$}}] (ab) {};
        \draw (b) -- (c) node [midway, circle,fill=black,inner sep=0pt,minimum size=5pt,label=right:{$G_e\times \ZZ_{q_4}$}] (bc) {};
        \draw (c) -- (d) node [midway, circle,fill=black,inner sep=0pt,minimum size=5pt,label=below right:{$G_w$}] (cd) {};
        \draw (d) -- (e) node [midway, circle,fill=black,inner sep=0pt,minimum size=5pt,label=below:{$G_e\times \ZZ_{q_5}$}] (de) {};
        \draw (e) -- (a) node [midway, circle,fill=black,inner sep=0pt,minimum size=5pt,label=left:{$G_v$}] (ea) {};
        \node [circle,fill=black,inner sep=0pt,minimum size=5pt, label={[right,xshift=0.34cm,yshift=0.125cm]:{$G_e$}}] (m) {};
        \draw (m) -- (a);
        \draw (m) -- (b);
        \draw (m) -- (c);
        \draw (m) -- (d);
        \draw (m) -- (e);
        \draw (m) -- (ab);
        \draw (m) -- (bc);
        \draw[dashed] (m) -- (cd);
        \draw (m) -- (de);
        \draw[dashed] (m) -- (ea);
\end{tikzpicture}
    \caption[A complex of groups over the pentagon.]{The left pentagon shows a labelling of the types $J\in\cals$.  The right pentagon shows the local groups after applying Thomas' functor to a graph of groups with a single edge.  In both pentagons the dashed line shows the embedding of the graph.  If the graph of groups has a single vertex, then $G_v=G_w$, $q_1=q_2$, $q_3=q_4$, the edge $(\{i_1,i_5\},\{i_1\})$ is glued to $(\{i_2,i_5\},\{i_2\})$, and the edge $(\{i_1,i_3\},\{i_1\})$ is glued to $(\{i_2,i_4\},\{i_2\})$.}
    \label{fig.pentagon}
\end{figure}
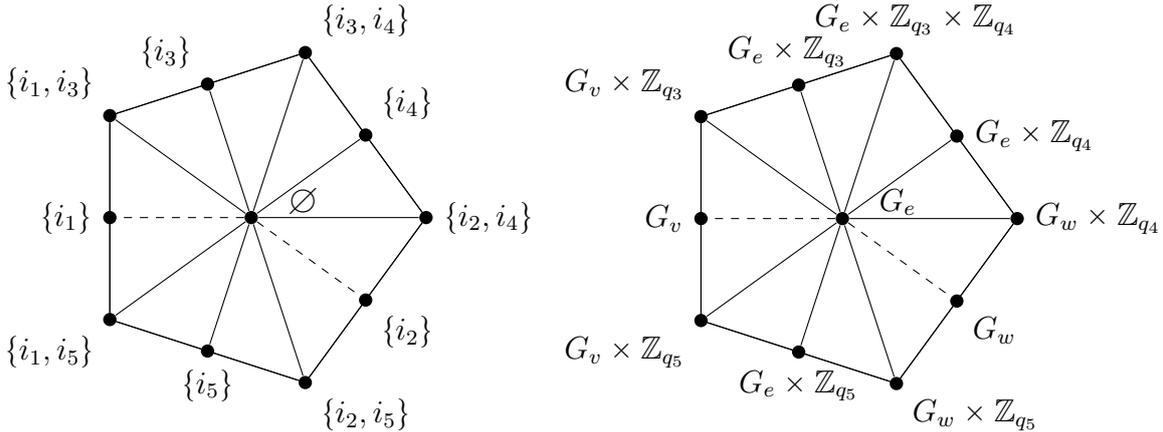

\begin{construction}[Thomas' Functor \cite{Thomas2006}] 
\textit{Let $X$ be a right-angled building of type $(W,I)$ and parameters $\{q_i\}$.  For each $i_1,i_2\in I$ such that $m_{i_1,i_2}=\infty$ let $\calt$ be the $(q_{i_1},q_{i_2})$-biregular tree.  Suppose \ref{Thomas.1} holds and if $q_{i_1}=q_{i_2}$ then \ref{Thomas.2} holds with $g$ an extension of $h$.  Then there is functor $F:\calg(\calt)\to\calc(X)$ preserving faithfulness and coverings.}

We will construct $F$ as a composite $F_2\circ F_1$.  We first define $F_1:\calg\to\calc_1$.  Let $(A,\cala)$ be a graph of groups and $|A|$ the geometric realisation of $A$.  We will construct a complex of groups $F_1(A)$ over $|A|$.  For the objects we have:
\begin{itemize}
    \item The local groups at the vertices of $|A|$ are the vertex groups of $\cala$.
    \item For all $e\in EA$ let $\sigma_e=\sigma_{\overline{e}}$ be the vertex of the barycentric subdivision $|A|'$ at the midpoint of $e$.
    \item The local group at $\sigma_e$ in $F_1(A)$ is $A_e=A_{\overline{e}}$.
    \item A monomorphism $\alpha_e:A_e\to A_{i(e)}$ in $A$ induces the same monomorphism in $F_1(A)$.
\end{itemize}
Let $\phi:A\to B$ be a morphism of graphs of groups over a map of graphs $f$, note that by \cite[Proposition~2.1]{Thomas2006} $F_1$ is not injective on morphisms.  We define $F_1(\phi)$ as follows:
\begin{itemize}
    \item The map $f$ induces a polyhedral map $f':|A|'\to|B|'$ so we will define $F_1(\phi):F_1(A)\to F_1(B)$ over $f$.
    \item Now take the morphisms on the local groups to be the same as for $\phi$.
%    \item Let $a\in E|A|'$, then the structure map $\psi_a$ in $F_1(A)$ is the same as some $\alpha_e$ in $A$
\end{itemize}
Let $\calc(\calt)=\im(F_1(\calg(\calt)))$ and $G(Y)\in\calc(\calt)$.  Now, we will define $F_2:\calc(\calt)\to\calc(X)$ as follows:
\begin{itemize}
    \item We first embed $Y'$ into a canonically constructed polyhedral complex $F_2(Y)$.  For each $e\in EY$ let $P_e'$ be a copy of $P'$ and identify the midpoint of $e$ with the cone vertex $x_0$ of $P_e'$.
    \item If $Y$ is $2$-colourable with colours $i_1$ and $i_2$ (from the valences of the Bass--Serre tree if $q_{i_1}\neq q_{i_2}$), then we identify the vertex of $e$ of type $i_j$ with the $i_j$-vertex of $P_e'$.
    \item Suppose $Y$ is not $2$-colourable.  If $e\in EY$ is not a loop in $Y$ then identify one vertex of $e$ with the $i_1$-vertex of $P_e'$ and the other with the $i_2$-vertex.  If $e$ forms a loop then we attach $P_e'/h$ (where $h$ is the isometry from the assumption) and identify the vertex of $e$ to the image of the $i_1$- and $i_2$-vertices of in $P_e'/h$.
    \item Glue together, either by preserving type on the $i_1$- and $i_2$-faces or by the isometry $h$, the faces of the the $P_e'$ and $P_e'/h$ whose centres correspond to the same vertex of $Y$.  Let $F_2(Y)$ denote the resulting polyhedral complex.
    \item Note that $Y'\rightarrowtail F_2(Y)$ and that each vertex of $F_2(Y)$ has a unique type $J\in\cals$ or two types $J$ and $h(J)$ where $i_1\in J\in\cals$ and $h$ is the isometry from the assumption.
    \item Fix the local groups and structure maps induced by the embedding of $Y'$ in $F(Y)$.  For each $i\in I$ let $G_i=\ZZ_{q_i}$ and for $J\subseteq I$ let $G_J=\prod_{j\in J}G_j$.  For each $e\in EY$ let $G_e$ be the local group at the midpoint of $e$.
    \item Let $J\in \cals$ such that neither $i_1$ or $i_2$ are in $J$.  The local group at a vertex of type $J$ is $G_e\times G_J$.  The structure maps between such local groups are the natural inclusions.
    \item Let $J\in\cals$ and suppose $i_k\in J$ for one of $k=1$ or $k=2$.  Since $m_{i_1,i_2}=\infty$ both $i_1$ and $i_2$ cannot be in $J$.  Let $F_e$ be the $i_k$-face of $P_e'$ or the glued face of $P_e'/h$.  The vertex of type $J$ in $P_e'$ or $P_e'/h$ is contained in $F_e$.  Let $v$ be the vertex of $Y$ identified with the centre of $F_e$ and let $G_v$ be the local group at $v$ in $G(Y)$
    \item The local group at the vertex of type $J$ is $G_v\times G_{J\backslash \{i_k\}}$.  For each $J'\subset J$ with $i_k\in J'$ the structure map $G_v\times G_{J'\backslash \{i_k\}}\rightarrowtail G_v\times  G_{J\backslash \{i_k\}}$ is the natural inclusion.  For each $J'\subset J$ with $i_k\not\in J'$ the structure map $G_e\times G_{J'}\rightarrowtail G_v\times  G_{J\backslash \{i_k\}}$ is the product of the structure map $G_e\rightarrowtail G_v$ in $G(Y)$ and the natural inclusion.
\end{itemize}
Now, let $\phi:G(Y)\to H(Z)$ be a morphism in $\calc(\calt)$ over a non-degenerate polyhedral map $f:Y\to Z$.  We will define $F_2(\phi)$ as follows:
\begin{itemize}
    \item If $Y$ and $Z$ are two colourable $f$ extends to a polyhedral map $F_2(f):F_2(Y)\to F_2(Z)$.  Otherwise we use \ref{Thomas.1} to construct $F_2(f)$.
    \item If $\tau\in VF(Y)$ then $G_\tau=G_\sigma\times G_J$ where $\sigma$ is a vertex of $Y'$.  The homomorphism of local groups $G_\sigma\times G_J\to H_{f(\sigma)}\times G_{J}$ is $\phi_\sigma$ on the first factor and the identity on the other factors.
    \item Let $a\in EF(Y)$.  If $\psi_a$, the structure map along $a\in F_2(G(Y))$, has a structure map $\psi_b$ from $G(Y)$ as its first factor, put $F_2(\phi)(b)=\phi(a)$.  Otherwise set $F_2(\phi)(b)=1$.
\end{itemize}
\end{construction}

We will now show the functor takes graphs of lattices to complexes of lattices and deduce a number of consequences.  For locally compact group $H$ let $\rm{Lat}(H)$ denote the set of $H$-lattices and let $\rm{Lat}_u(H)$ denote the set of uniform $H$-lattices.

\begin{thm}\label{thm.functor.new}
Let $Y$ be a right-angled building of type $(W,I)$ and parameters $\{q_i\}$ and let $A=\Aut(Y)$.  For each $i_1,i_2\in I$ such that $m_{i_1,i_2}=\infty$ let $\calt$ be the $(q_{i_1},q_{i_2})$-biregular tree and let $T=\Aut(\calt)$.  Suppose \ref{Thomas.1} holds and if $q_{i_1}=q_{i_2}$ then \ref{Thomas.2} holds with $g$ an extension of $h$, and let $F:\calg(\calt)\to\calc(Y)$ be Thomas' functor.  Let $X$ be a finite dimensional proper $\CAT(0)$ space and assume $H=\Isom(X)$ contains a cocompact lattice.  The following conclusions hold:
\begin{enumerate}
    \item If $G(\calt)$ is a graph of $H$-lattices, then $F(G(\calt))$ is a complex of $H$-lattices. \label{thm.functor.gtoc}
    \item $F$ induces an inclusion of sets $\rm{Lat}_u(H\times T)\rightarrowtail\rm{Lat}_u(H\times A)$. \label{thm.functor.uniform}
    \item If $Y$ is a $\CAT(0)$ polyhedral complex then $F$ induces an inclusion of sets $\rm{Lat}(H\times T)\rightarrowtail \rm{Lat}(H\times A)$. \label{thm.functor.nonuni}
\end{enumerate}
Let $\Gamma$ be a uniform $(H\times T)$-lattice and let $F\Gamma$ be the corresponding $(H\times A)$-lattice.
\begin{enumerate}[resume]
    \item $\pi_T(\Gamma)$ is discrete if and only if $\pi_A(F\Gamma)$ is discrete.  Moreover, $\pi_H(\Gamma)=\pi_H(F\Gamma)$. \label{thm.functor.proj}
    \item If $\Gamma$ satisfies any of $\{$algebraically irreducible, non-residually finite, not virtually torsion free$\}$, then so does $F\Gamma$. \label{thm.functor.props}
\end{enumerate}
\end{thm}
\begin{proof}
We first prove \eqref{thm.functor.gtoc}.  We will first verify the conditions on the local groups and then construct a morphism to $H$.  Let $(B,\calb,\psi)$ be a graph of $H$-lattices and consider the image $L(Z)$ of $\calb$ under $F$.  Here $Z=F(B)$.  Each local group in $L(Z)$ is of the form $G_\sigma\times G_J$ where $G_\sigma$ is a local group in $\calb$ and $G_J$ is a finite product of finite cyclic groups.  We have a morphism $\psi:\calb\to H$ such that the image of each local group $G_\sigma$ is an $H$-lattice and the restriction to $G_\sigma$ has finite kernel.  Thus, by construction the local groups in $L(Z)$ are commensurable in $\pi_1(L(Z))$.  We define $F(\psi_\sigma)$ to be the composite $\psi|_{G_\sigma}\circ\pi_\sigma:G_\sigma\times G_J\twoheadrightarrow G_\sigma\to\psi(G_\sigma)$, thus commensurability of the images in $H$ is immediate.

We will now deal with the edges.  Note the twisting elements in $L(Z)$ are all trivial and the complex of groups $H$ has all structure maps the identity.  Let the structure maps in $L(Z)$ be denoted by $\lambda_a$ for $a\in EZ'$ and the structure maps in $\calb$ by $\alpha_e$ for $e\in EB$.  The family of elements $(t_e)_{e\in EB}$ in the path group $\pi(\calb)$ are mapped under $\psi$ to elements of $\Comm_H(\psi(G_\sigma))$ where $G_\sigma$ is some local group.  Now, let $a\in EZ'$, then by construction $a$ either corresponds to a subdivision of an edge $a$ in $EB$ in which case we define $(F\psi)(a)=\psi(a)$.  Or, $a$ corresponds to a inclusion of local groups $G_\sigma\times G_{J'}\to G_\sigma\times G_J$, in which case we define $(F\psi)(a)=1_H$.  

It remains to verify the two edge axioms for a morphism.  For each $a\in EZ'$ corresponding to the subdivision of an edge $a$ in $EB$ we have \[\Ad((F\psi)(a))\circ F(\psi_{i(a)})=\Ad(\psi(a))\circ \psi_{i(a)}\circ \pi_a=\psi_{t(a)}\circ\alpha_a\circ \pi_a=F(\psi_{t(a)})\circ F(\alpha_a),\]
where $\pi_a$ is the surjection $G_a\times G_J\twoheadrightarrow G_a$.  For any other edge $a\in EZ'$ we have
\[\Ad((F\psi)(a))\circ F(\psi_{i(a)})=F(\psi_{i(a)})\text{ and }F(\psi_{t(a)})\circ\lambda_a=F(\psi_{i(a)}).\]
Finally, the other condition that $(F\psi)(ab)=(F\psi)(a)(F\psi)(b)$ for $(a,b)\in E^2Z'$ is verified trivially.  Thus, $F(\calb)=L(Z)$ is a complex of $H$-lattices. $\blackdiamond$

We will next prove \eqref{thm.functor.uniform}.  Let $\Gamma$ be an $(H\times T)$-lattice.  By Theorem~\ref{thm.structure}, $\Gamma$ splits as graph of $H$-lattices $\calb$.  Thus, by \eqref{thm.functor.gtoc} we obtain a complex of $H$-lattices $F(\calb)$ with fundamental group $\Lambda$.  By Theorem~\ref{thm.col}\eqref{thm.col.uniform} it suffices to show that for each local group $G_\sigma$ in $F(\calb)$ the kernel $K_\sigma=\Ker(\pi_H|_{FG_\sigma})$ acts faithfully on $X$.  Now, $K_\sigma$ is a direct product of $L_\sigma=\Ker(\pi_H|_{G_\sigma})$ with a direct product of cyclic groups $G_J$, where $G_\sigma$ is a local group in $\calb$.  By construction $G_J$ acts faithfully on $X$ and by Theorem~\ref{thm.structure}, $K_\sigma$ acts faithfully on $\calt$ which embeds into $Y$.  In particular, $K_\sigma$ acts faithfully on $X$. $\blackdiamond$

We will next prove \eqref{thm.functor.nonuni}.  We construct a complex of lattices as in the previous case.  The proof for \eqref{thm.functor.nonuni} is now identical once we have verified that covolume condition in Theorem~\ref{thm.col}\eqref{thm.col.nonuniform}.  Let $c$ denote the covolume of an $(H\times T)$-lattice $\Gamma$ with associated graph of lattices $(B,\calb)$, this is given by the formula $c=\sum_{\sigma\in VA}\mu(\Gamma_\sigma)<\infty$.  Now, every vertex of the complex $Z=F(B)$ has local group isomorphic to a finite extension of some $\Gamma_\sigma$.  In particular we may bound $\sum_{\sigma\in Z}\mu(\Gamma_\sigma)$ by $\ell\times c$ where $\ell$ is the number of vertices in the finite Coxeter nerve of $X$. $\blackdiamond$

The proof of \eqref{thm.functor.proj} follows from the proof of \eqref{thm.functor.gtoc}. $\blackdiamond$

The proof of \eqref{thm.functor.props} follows from either applying Theorem~\ref{thm.CMirrCrit} to \eqref{thm.functor.proj} (algebraically irreducible) or the fact $\Gamma\rightarrowtail F\Gamma$ and the properties of residual finiteness and virtual torsion-freeness are subgroup closed. $\blackdiamond$
\end{proof}

\subsection{Examples and applications}

In this section we will detail some sample examples and applications of the functor theorem.

We can obtain a number of examples by applying Thomas' functor to any irreducible $(\Isom(\EE^n)\times T)$-lattice.  This will give a non-biautomatic group acting properly discontinuously cocompactly on $\EE^n\times X$ where $X$ is a sufficiently symmetric right-angled building.  More precisely, we have the following corollary:

\begin{corollary}[General version of Corollary~\ref{corx.thomas.notbiaut}]\label{cor.thomas.notbiaut}
Let $Y$ be a right-angled building of type $(W,I)$ and parameters $\{q_i\}$ and let $A=\Aut(Y)$.  For each $i_1,i_2\in I$ such that $m_{i_1,i_2}=\infty$ let $\calt$ be the $(q_{i_1},q_{i_2})$-biregular tree and let $T=\Aut(\calt)$.  Suppose \ref{Thomas.1} holds and if $q_{i_1}=q_{i_2}$ then \ref{Thomas.2} holds with $g$ an extension of $h$ and let $F:\calg(\calt)\to\calc(Y)$ be Thomas' functor.  Let $\Gamma$ be a uniform $(\Isom(\EE^n)\times T)$-lattice and suppose $\pi_{\OO(n)}(\Gamma)$ is infinite.  Then, $F\Gamma$ is a uniform $(\Isom(\EE^n)\times A)$-lattice which is not virtually biautomatic nor residually finite.  In particular, if $Y$ is irreducible, then the direct product of a uniform $A$-lattice with $\ZZ^2$ is not quasi-isometrically rigid.
\end{corollary}
\begin{proof}
By Theorem~\ref{thm.functor.new} $F\Gamma$ is a uniform $(\Isom(\EE^n)\times A)$-lattice with a non-discrete projection to $\OO(n)$.  That $F\Gamma$ is not virtually biautomatic then follows from Theorem~\ref{thm.notbiaut}.  The failure of quasi-isometric rigidity follows from the fact that the direct product of a uniform $A$ lattice with $\ZZ^2$ is reducible, whereas, the weakly irreducible lattice is algebraically irreducible by Theorem~\ref{thm.CMirrCrit} and so does not virtually split as a direct product of two infinite groups.  In particular, the groups cannot by virtually isomorphic.
\end{proof}

\begin{example} \label{ex.LMX.pres}
Let $\Gamma=\LM(A)$ where $A$ is the matrix corresponding to the Pythagorean triple $(3,4,5)$ .  Recall the group acts on $\EE^2\times\calt_{10}$.  Let $X$ be the right angled building whose Coxeter nerve is the regular pentagon and whose parameters are given by $q_1=q_2=10$, $q_3=q_4=k$, and $q_5=\ell$.  Let $A$ be the automorphism group of $X$ and consider $F\Gamma$ the image of $\Gamma$ under Thomas' functor $F$ as in Figure~\ref{fig.pentagon}.  By Theorem~\ref{thm.functor.new}, the group $F\Gamma$ is a non-residually finite $(\Isom(\EE^n)\times A)$-lattice with non-discrete projections to both factors and is irreducible as an abstract group.  Moreover, by the previous corollary, $F\Gamma$ is not virtually biautomatic.

We will now construct a presentation for $\Lambda_{k,\ell}:=F\Gamma$.  The group has generators $a,b,x_3,x_4,x_5,t$ and relations
\[x_3^{k}=x_4^k=x_5^\ell=1,\ [a,b],\ [a,x_3],\ [a,x_4],\ [a,x_5],\ [b,x_3],\ [b,x_4],\ [b,x_5],\ [x_3,x_4], \]
\[ta^2b^{-1}t^{-1}=a^2b,\ tab^2t^{-1}=a^{-1}b^2,\ tx_3t^{-1}=x_4, [t,x_5]. \]
\end{example}

\begin{prop}
The group $\Lambda_{2,2}$ in Example~\ref{ex.LMX.pres} is virtually torsionfree.  This is witnessed by the index $16$ subgroup \[\Delta:=\langle a,\ b,\ x_3tx_4t^{-1},\ x_3x_4t^{-2},\ (x_5x_3)^2,\ (x_5x_4)^2,\ t^{-1}x_3x_4t^{-1},\ (tx_5x_4t^{-1})^2\rangle.\]
\end{prop}
\begin{proof}
  The quotient $\Lambda_{2,2}/\Delta$ is isomorphic to $D_4\times \ZZ_2$ which has order $16$.  By construction every torsion element of $\Lambda_{2,2}$ is conjugate to some power of $x_3$, $x_4$, $x_5$ or $x_3x_4$.  Indeed, every torsion element is contained in a vertex or edge stabiliser of the action on the pentagonal building and acts trivially on $\EE^2$. Each of these elements is mapped to a non-trivial element of $D_4\times\ZZ_2$.  In particular, the kernel $\Delta$ is torsion-free.
\end{proof}

\begin{example}
Let $n\geq 2$ and let $\Gamma_n$ be the irreducible lattice constructed in Example~\ref{ex.LM.odddim.lat} acting on $\EE^n\times\calt_{10n}$.  Let $X$ be a right angled building satisfying \ref{Thomas.1} and \ref{Thomas.2} with automorphism group $A$ and parameters $\{q_i \}$ all equal to $10n$.  Applying Thomas' functor and Theorem~\ref{thm.functor.new} to $\Gamma_n$ we obtain a non-residually finite $(\Isom(\EE^n)\times A)$-lattice with non-discrete projections to both factors.  Moreover by Corollary~\ref{cor.thomas.notbiaut}, $\Gamma_n$ is not virtually biautomatic.
\end{example}

We will now show the existence of non-residually finite lattices in arbitrary products of sufficiently symmetric isometric and non-isometric right-angled buildings.  We note that Bourdon's ``hyperbolization of Euclidean buildings" \cite[Section~1.5.2]{Bourdon2000} can be used to construct weakly irreducible uniform lattices in products of hyperbolic buildings.  We will provide a number of examples to show that the groups we construct here are distinct.

\begin{corollary}\label{thm.treelatstobuildings}
Let $\Gamma$ be a weakly irreducible lattice in product of trees $\calt_1\times\dots\times \calt_n$ such that $\calt_k$ is $(t_{k_1},t_{k_2})$-biregular.  Let $X_1\times\dots\times X_n$ be a product of irreducible right angled buildings satisfying \ref{Thomas.1} and \ref{Thomas.2}.  Suppose $X_k$ is of type $(W_k,I_k)$, has parameters $\{t_{k_1},t_{k_2},q_{k_3},\dots,q_{k_{n_k}}\}$ where $m_{k_{i_1},k_{i_2}}=\infty$ and $A_k=\Aut(X_k)$.  Then, the lattice $\Lambda=F^n\Gamma$ obtained by applying Thomas' functor $n$ times (once for each tree $\calt_k$ corresponding to the building $X_k$) is a lattice in $A_1\times\dots\times A_n$, is irreducible, and is non-residually finite.
\end{corollary}
\begin{proof}
Let $T_k=\Aut(\calt_k)$.  The result follows from applying Theorem~\ref{thm.functor.new} $n$ times as follows.  Consider $\Gamma$ as a graph of $(T_2\times\dots\times T_n)$-lattices and apply $F$ to obtain a $(A_1\times T_2\times\dots\times T_n)$-lattice with the desired properties (non-residual finiteness follows from the fact that the projection to $T_2\times\dots\times T_n$ has a non-trivial kernel).  Now, consider $F\Gamma$ as a graph of $(A_1\times T_3\times\dots T_n)$-lattices and proceed by induction on the index $k$.
\end{proof}

\begin{example} In \cite{BurgerMozes2000lat,BurgerMozes1997} Burger and Mozes construct for each pair of sufficiently large even integers $(m,n)$ a finitely presented simple group as a uniform lattices in a product of trees $\calt_m\times\calt_n$ (for more examples see \cite{Rattaggi2007a,Rattaggi2007b,Radu2020}).  Applying Theorem~\ref{thm.treelatstobuildings}, we obtain uniform non-residually finite algebraically and weakly irreducible lattices acting on a product of buildings $X_1\times X_2$ each satisfying \ref{Thomas.1} and \ref{Thomas.2} with $X_1$ having some parameters equal to $m$ and $X_2$ having some parameters equal to $n$.
\end{example}

\begin{example}
In \cite{Hughes2021b} the author constructed a non-virtually torsion-free irreducible uniform lattice in a product of two locally finite trees.  Applying Theorem~\ref{thm.treelatstobuildings}, we obtain uniform non-residually finite non-virtually torsion-free lattices acting on a product of buildings.
\end{example}

\section{Some questions}
In every example of an $(\Isom(\EE^n)\times T)$-lattice known to the author, the lattice is virtually torsion-free.

\begin{question}
Are there non-virtually torsion-free $(\Isom(\EE^n)\times T)$-lattices?
\end{question}

Since it is possible to characterise $(\Isom(\EE^n)\times T)$-lattice in terms of $C^\ast$-simplicity, it would be interesting to recover the characterisation for complexes of $\Isom(\EE^n)$-lattices.

\begin{question}
Are the algebraically irreducible non-biautomatic groups constructed in and Section~\ref{sec.functor} $C^\ast$-simple?
\end{question}

More generally we ask:

\begin{question}
When is a $\CAT(0)$ lattice $C^\ast$-simple?
\end{question}

The characterisation of algebraically irreducible $(\Isom(\EE^n)\times T)$-lattices in terms of $C^\ast$-simplicity suggests the following question:

\begin{question}
Can $C^\ast$-simplicity of a Leary-Minasyan group $\LM(A)$ be determined by properties of the matrix $A$?
\end{question}

It remains open whether there are irreducible lattices in products of Salvetti complexes and a Euclidean space.

\begin{question}
    For which flag complexes $L$ and positive integers $n$ are there irreducible uniform lattices in the product $\Aut(X_L)\times\Isom(\EE^n)$?
\end{question}

\bibliographystyle{halpha}
\bibliography{refs.bib}

\end{document}